\documentclass[11pt,a4paper]{article}

\title{Dehn-Seidel Twist, $C^0$ symplectic topology and barcodes}
\author{Alexandre Jannaud }
\date{January 2021}

\usepackage[utf8]{inputenc}
\usepackage[T1]{fontenc}

\usepackage[a4paper]{geometry}
\usepackage[frenchb,english]{babel}

\usepackage{amsthm}

\usepackage{fancyhdr}
\usepackage{import}
\usepackage[final]{graphicx}
\usepackage{numprint} 
\usepackage{adjustbox}
\usepackage{varioref}
\usepackage{hyperref}
\usepackage{nameref}

\usepackage{caption}
\usepackage{yfonts}
\usepackage{thmtools}

\usepackage{setspace}
\usepackage{amsmath,amssymb,latexsym,amsthm,epsfig,stmaryrd}
\usepackage{tikz-cd}
\usepackage{mathrsfs}
\usepackage[all]{xy} 
\usepackage{multirow} 

\usepackage[capitalize,nameinlink,noabbrev]{cleveref}

\allowdisplaybreaks[1]

\usepackage{calc}

\newcommand{\D}{\mathbb{D}}
\newcommand{\R}{\mathbb{R}}
\newcommand{\C}{\mathbb{C}}
\newcommand{\K}{\mathbb{K}}
\newcommand{\Z}{\mathbb{Z}}
\newcommand{\T}{\mathbb{T}}
\newcommand{\N}{\mathbb{N}}

\newcommand{\Ham}{\mathrm{Ham}}
\newcommand{\Symp}{\mathrm{Symp}}
\newcommand{\Diff}{\mathrm{Diff}}

\newcommand{\Spec}{\mathrm{Spec}}

\def\Id{\mathrm{Id}}

\DeclareMathOperator{\im}{Im}

\DeclareRobustCommand\openone{\leavevmode\hbox{\small1\normalsize\kern-.33em1}}

\def\calA{\mathcal{A}}
\def\calB{\mathcal{B}}

\def\calD{\mathcal{D}}

\def\calI{\mathcal{I}}
\def\calL{\mathcal{L}}

\newtheorem{thm}{Theorem}[section]
\newtheorem{lemma}[thm]{Lemma}
\newtheorem{prop}[thm]{Proposition}
\newtheorem{dfn}[thm]{Definition}
\newtheorem{cor}[thm]{Corollary}
\newtheorem{conj}[thm]{Conjecture}

\newtheorem{question}{Question}

\newtheorem{thmi}{Theorem}

\newtheorem{thmx}{Theorem}
\crefname{thmx}{Theorem}{Theorems}

\newtheorem{corx}[thmx]{Corollary}
\crefname{corx}{Corollary}{Corollaries}

\theoremstyle{definition}
\newtheorem{rk}[thm]{Remark}

\theoremstyle{definition}
\newtheorem{rki}[thmi]{Remark}

\begin{document}

\maketitle

\begin{abstract}
 We initiate the study of the $C^0$ symplectic mapping class group, i.e. the group of isotopy classes of symplectic homeomorphisms. 
We prove that none of the different powers of the square of the Dehn-Seidel twist belong to the same connected component of the group of symplectic homeomorphisms of certain Liouville domains. This generalizes to the $C^0$ setting a celebrated result of Seidel. In other words, we obtain the non-triviality of the $C^0$ symplectic mapping class group in these domains and in fact an element of infinite order.

For that purpose, we develop a method coming from Floer theory and the theory of barcodes. This builds on recent developments of $C^0$-symplectic topology. In particular, we adapt and generalize to our context results by Buhovsky-Humilière-Seyfaddini and Kislev-Shelukhin.
\end{abstract}
\setcounter{tocdepth}{1}
\tableofcontents
\section*{Introduction}\label{introduction}

\addcontentsline{toc}{section}{Introduction}

We will work with a symplectic manifold $(M^{2n},\omega)$. The group of symplectomorphisms will be denoted $\Symp(M,\omega)$, the group of symplectomorphisms isotopic to the identity in $\Symp(M,\omega)$ will be denoted by $\Symp_0(M,\omega)$ and the group of Hamiltonian diffeomorphisms $\Ham(M,\omega)$.

\subsubsection*{$C^0$ symplectic topology}

$C^0$ symplectic topology was born with the famous Gromov-Eliashberg theorem \cite{El87} stating that if a sequence of symplectomorphisms $C^0$-converges to a diffeomorphism, then this diffeomorphism is a symplectomorphism as well.

Considering this theorem, \emph{symplectic homeomorphisms} were naturally defined as the $C^0$-closure of symplectomorphisms.
\begin{dfn}
Let $(M,\omega)$ be a symplectic manifold. A homeomorphism $\varphi$ of $M$ is called a \emph{symplectic homeomorphism} if it is the uniform limit of a sequence of symplectic diffeomorphisms.
\end{dfn}
The main goal in $C^0$-symplectic topology is then to understand whether it is possible or not to do symplectic topology with continuous objects. By $C^0$-topology we mean the topology induced by, for $\varphi$ and $\psi$ homeomorphisms,
$$d(\varphi,\psi)=\max\left\{\sup\limits_{p\in M}d(\varphi(p),\psi(p)),\quad\sup\limits_{p\in M}d(\varphi^{-1}(p),\psi^{-1}(p))\right\},$$
for an arbitrary Riemannian metric $d$ in $M$.

Laudenbach and Sikorav \cite{LaudSik94} proved an analogue of the Gromov-Eliashberg theorem, but with Lagrangian submanifolds replacing symplectomorphisms. 

More than a decade later, $C^0$-symplectic topology took a step forward, when Oh and Müller \cite{MulOh07} introduced a notion of Hamiltonian homeomorphisms, which they called hameomorphisms. These maps have the property of being generated in some sense by continuous Hamiltonians, hence appearing as good $C^0$ generalizations of Hamiltonian diffeomorphisms. This notion renewed the interest for $C^0$ symplectic topology.

More recently, $C^0$ symplectic topology took a second step forward. Humilière-Leclercq-Seyfaddini proved a result of coisotropic rigidity in \cite{HLS15} and a reduction result in \cite{HLS16}, both papers proving that, in many aspects, symplectic homeomorphisms tend to behave as symplectic diffeomorphisms. At the same time, Buhovsky-Opshtein \cite{BO16} exhibited, among other rigidity results, the first flexibility behaviour for symplectic homeomorphisms: a symplectic homeomorphism leaving invariant a smooth symplectic submanifold $V$, and whose restriction to $V$ is smooth but not symplectic. It was shortly followed by the counter-example to the Arnold conjecture by Buhovsky-Humilière-Seyfaddini \cite{BHS18}, which is another beautiful example of $C^0$-symplectic flexibility. 

In parallel, much progress has been made regarding the barcodes and action selectors which are the main tools used to study these homeomorphisms. The main results concern the $C^0$-continuity for action selectors, started by Seyfaddini \cite{Sey13} with his $\varepsilon$-shift trick, and followed by Buhovsky-Humilière-Seyfaddini \cite{BHS19}. Seyfaddini \cite{Sey13}, Buhovsky-Humilière-Seyfaddini \cite{BHS19}, Kawamoto \cite{Kaw19}, Shelukhin \cite{Shel18,Shel19} proved the $C^0$-continuity of the action selectors in various settings. Using a result of Kislev-Shelukhin \cite{KS18}, this implies the $C^0$-continuity of barcodes in the same settings. Le Roux-Seyfaddini-Viterbo \cite{LRSV18} proved the continuity of barcodes for Hamiltonians on surfaces, without using Kislev-Shelukhin's result.

\subsubsection*{Dehn twists and mapping class groups}

\bigskip
\begin{figure}[h]
    \centering
    \label{fig Dehn Twist}
\includegraphics[scale=0.5]{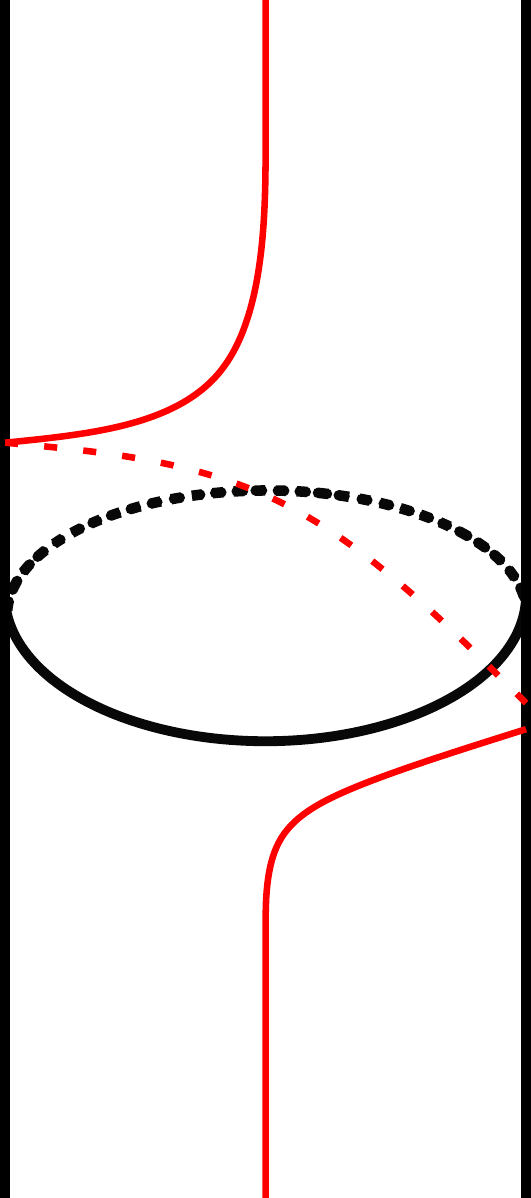}
\caption{Dehn Twist in $T^*S^1$.\\ The red curve represents the image by the Dehn twist of a fiber of $T^*S^1$.}
\centering
\end{figure}

Dehn twists are diffeomorphisms supported in the neighbourhood of a simple loop in surfaces.

Let us first describe the local model. We consider the annulus $S^1\times [-1,1]=T^*_1S^1$. We denote $\tau:T_1^*S^1\rightarrow T_1^*S^1$ the map given by
$$\tau(\theta,t)=(\theta+2\pi f(t),t),$$
where $f:[-1,1]\rightarrow \R^+$ is a smooth function equal to $0$ near $-1$ and equal to $1$ near $1$.
This map is called a twist map.
Now that we have our model, we can describe the Dehn twist for surfaces. It consists of a map which agrees with our local model on the neighbourhood of a given loop $l$ and is equal to the identity away from this loop. It is called the \emph{Dehn twist along $l$} and it is denoted $\tau_l$.
One can prove that the isotopy class of $\tau_l$ only depends on the isotopy class of $l$. If the loop along which the Dehn twist is defined is not contractible, then the Dehn twist is not isotopic to the identity.

The Dehn twists are of particular interest when studying the mapping class group of surfaces. Let us recall that the \emph{mapping class group} is defined, in the case of a smooth oriented manifold $M$ by
$$\mathrm{MCG}(M)=\pi_0(\Diff^+(M)).$$
Let $\Sigma$ be an oriented smooth surface and denote $\omega$ an associated symplectic form on $\Sigma$. We denote by $\mathrm{MCG}^{\omega}(\Sigma)$ the mapping class group for area-preserving diffeomorphisms. This $\mathrm{MCG}^{\omega}(\Sigma)$ is nothing but $\pi_0(\Symp(\Sigma,\omega))$. Let us also denote by $\mathrm{MCG}(\Sigma,C^0)=\pi_0(\mathrm{Homeo}^+(\Sigma))$ the mapping class group for homeomorphisms and by $\mathrm{MCG}^{\omega}(\Sigma,C^0)=\pi_0(\mathrm{Homeo}^{+,\omega}(\Sigma))$ the mapping class group for area-preserving homeomorphisms.

One can prove that the mapping class group $\mathrm{MCG}(\Sigma)$ is generated by Dehn twists. We actually have the following isomorphisms:
\begin{equation}\label{iso MCG}
\mathrm{MCG}^{\omega}(\Sigma)\cong \mathrm{MCG}(\Sigma) \cong \mathrm{MCG}(\Sigma,C^0) \cong \mathrm{MCG}^{\omega}(\Sigma,C^0).    
\end{equation}
The first isomorphism is a consequence of Moser's trick. The surjectivity of the second one comes from the fact that any homeomorphism is a limit of diffeomorphisms, which is for example proven in \cite{LRSV18}, together with the fact that the group of homeomorphisms is locally contractible \cite{Che08}. Its injectivity comes from the local contractibility of the group of diffeomorphisms \cite{LR10}. Finally, the third isomorphism is due to Fathi \cite{Fat80}.

In symplectic geometry, the mapping class group we are interested in is of course related to symplectomorphisms:
$$\mathrm{MCG}^{\omega}(M)=\pi_0(\Symp(M,\omega)).$$

These Dehn twists have been generalized to higher dimensions by Arnold \cite{Arn95} and they have been then intensively studied by Seidel in his PhD thesis \cite{SeiPhD} and in \cite{Seidel99,Seidel00,Seidel01}. We call these higher dimensional maps \emph{generalized Dehn twist}, or \emph{Dehn-Seidel twists}. They are defined in the neighbourhood of a Lagrangian sphere $L$, and thus will be denoted $\tau_L$. Let us briefly present Seidel's description of these maps. As in dimension $2$, we start by describing a local model in the cotangent bundle of a sphere. We denote
$$T^*_1S^n=\{\xi\in T^*S^n,\quad |\xi|\leq 1\},$$
where $|\cdot|$ denotes the dual of the standard round metric on $S^n$. In coordinates we have
$$T_1^*S^n=\{(u,v)\in\R^{n+1}\times\R^{n+1},~~|u|\leq 1,|v|=1,\langle u,v\rangle=0\},$$
and $\omega_{T_1^*S^n}=\sum_i du_i\wedge dv_i$.
We set
$$\sigma_t(u,v)=\left(\cos(2\pi t)u-\sin(2\pi t)v |u|, \cos(2\pi t)v+\sin(2\pi t)\frac{u}{|u|}\right),$$
for $t\in[0,1]$, and $(u,v)\in T^*S^n\setminus S^n$. When $t=1/2$, $\sigma$ corresponds to the antipodal map: $\sigma_{1/2}(u,v)=(-u,-v)$. Note that this antipodal map extends continuously to the zero-section. We choose a cut-off function $\rho:[0,1]\rightarrow\R$ such that $\rho$ is equal to $\frac{1}{2}$ near $0$ and equal to $0$ near $1$. We can now define $\tau$ by
$$\tau(\xi)=\sigma_{\rho(|\xi|)}(\xi).$$
This map is a symplectomorphism equal to the antipodal map on the zero-section and equal to the identity near the boundary of $T_1^*S^n$. When $n=1$, it is isotopic to the model Dehn twist on surfaces described above.

We now want to embed our local model into a symplectic manifold, matching the zero-section with a Lagrangian sphere. Let $(M,\omega)$ be a symplectic manifold with boundary, together with a Lagrangian embedding $l:S^n\rightarrow M$. Using Weinstein's neighbourhood theorem, we may implant this local model in the neighbourhood of the Lagrangian sphere $l(S^n)=L$. The isotopy class in $\Symp(M,\omega)$ of the resulting map $\tau_l$ only depends on $l$. This map is called \emph{the generalized Dehn}, or \emph{Dehn-Seidel twist along $l$}.

In his PhD thesis \cite{SeiPhD}, Seidel proved that in dimension $4$, the square of a  Dehn-Seidel twist is isotopic to the identity through smooth diffeomorphisms but is not through symplectomorphisms. He later generalized the last part of this result to higher dimensions using the technology of Lagrangian Floer cohomology in \cite{Seidel00}.

Using Seidel's notations, we start by describing what an $(A_k)$-configuration is. Let $M$ be a $2n$-dimensional compact symplectic manifold.
\begin{dfn}\label{def Ak}
An $(A_k)$-configuration in $M$ is a family of Lagrangian spheres $(l_1,... l_k)$ with images $(L_1,...L_k)$ such that
\begin{itemize}
    \item they are pairwise transverse
    \item for $2 \le j\le k-1$, $|L_i\cap L_{j}|=1$ if $i=j\pm 1$ and $|L_i\cap L_j|=\emptyset$ else.
\end{itemize}
\end{dfn}

Seidel proved \cite{Seidel99} that the affine hypersurface $(H,\omega)$ in $\C^{n+1}$ equipped with the standard symplectic form satisfying the equation
$$z_1^2+z^2_2+\cdots+z^2_n=z_{n+1}^{m+1}+\frac{1}{2}$$
contains an $(A_m)$-configuration of Lagrangian $n$-spheres. The name comes from the fact that these hypersurfaces are the Milnor fibres of type $(A_m)$ -singularities.

Following Seidel's paper \cite{Seidel99}, we briefly describe these Lagrangian spheres for $n=2$. Let us denote $\pi:H\rightarrow \C^2$ the projection onto the $(z_1,z_2)$ complex plane and $\sigma$ the map defined by $\sigma(z_1,z_2,z_3)=(z_1,z_2,e^{2i\pi/(m+1)}z_3)$. The projection is an $(m+1)$-fold covering branched along $C=\{z_1^2+z_2^2=\frac{1}{2},\quad (z_1,z_2)\in \C^2\}$ whose covering group is generated by $\sigma$. We now consider the map $f:S
^2\subset\R^3\rightarrow \C^2$ defined by
$$f(x_1,x_2,x_3)=(x_2(1+ix_1),x_3(1+ix_1)).$$ For all $x\in S^2$, we have $f(x)\in \C^2\setminus C$. This map is an immersion with one double point: $f(1,0,0)=f(-1,0,0)$. Let us denote $\tilde{f}:S^2\rightarrow H$ a lift of $f$. One can show that $\tilde{f}(S^2)$ and $\sigma\circ\tilde{f}(S^2)$ have only one intersection point, at $\tilde{f}(1,0,0)$. In the same way, the family
$$(\tilde{f}(S^2),\sigma\tilde{f}(S^2),...,\sigma^{m-1}\tilde{f}(S^2))$$
satisfies the intersection conditions of the previous definition. Finally one can choose a $2$-form $\omega_0$ on $H$, diffeomorphic to $\omega$, such that these spheres are Lagrangians, and thus $(H,\omega_0)$ admits a $(A_m)$-configuration. The fact that these two $2$-forms are diffeomorphic tells that $(H,\omega)$ contains such a configuration as well.

This allows us to have such configurations inside a Liouville domain. These objects were intensively studied by Khovanov-Seidel \cite{KovSeid00}, Seidel-Thomas \cite{ST01}, Seidel \cite{Sei08}, Keating \cite{Kea12}...

The theorem of Seidel which interests us in this paper is the following (\cite{Seidel00})
\begin{thmi}[Seidel \cite{Seidel00}]\label{thm Seidel intro}
Let $(M^{2n},\omega)$ be a compact symplectic manifold with contact type boundary, with n even, which satisfies $[\omega]=0$ and $2c_1(M,\omega)=0$. Assume that $M$ contains an $A_3$-configuration $(l_{\infty},l',l)$ of Lagrangian spheres.  Then $M$ contains infinitely many symplectically knotted Lagrangian spheres. More precisely, if one defines $L'^{(k)}=\tau_l^{2k}(L')$ for $k\in \Z$, then all the $L'^{(k)}$ are isotopic as smooth submanifolds of M, but no two of them are isotopic as Lagrangian submanifolds.
\end{thmi}
Here, $c_1(M,\omega)$ denotes the first Chern class of the tangent bundle $TM$. This theorem immediately implies that $\tau^2_L$ is not isotopic to the identity in $\Symp(M,\omega)$. Historically, this is the first higher dimensional result on the symplectic mapping class group.

\begin{rki}
In Seidel's theorem, it is assumed that $n$ is even. Indeed, for $n$ odd, one can prove that the square of the Dehn-Seidel twist acts non-trivially on homology making the previous result irrelevant. However, in the same paper \cite{Seidel00}, Seidel also proved an odd-dimensional counterpart of this theorem in which one should consider a composition of non-isotopic Dehn-Seidel twists.
\end{rki}

Seidel's result is deeply related to Picard-Lefschetz theory and thus to homological mirror symmetry. Nevertheless, this is an entirely different subject that will not be addressed here. However, many progress have been made on more related topics. For example Evans \cite{EV11} and Li-Li-Wu \cite{LLW15} showed that the symplectic mapping class group of some specific blow-ups of $\C \mathbb{P}^2$ is generated by Dehn-Seidel twists. Khovanov-Seidel \cite{KovSeid00} and Seidel-Thomas \cite{ST01}  proved that if two Lagrangian spheres intersect transversely at a single point, their associated Dehn twists satisfy a braid relation. This result was generalized in \cite{Kea12} by Keating for more general pairs of Lagrangians. In some specific cases Evans \cite{EV11} and Wu \cite{Wu13} proved that there is a weak homotopy equivalence between the group of compactly supported symplectomorphisms and a braid group on the disk. Moreover, Dimitroglou-Rizell and Evans \cite{DGE14} constructed from Dehn twists non-contractible families of symplectomorphisms.

As shown by Seidel's result, these questions are closely related to Lagrangian isotopy questions. For instance Coffey \cite{Cof05} showed that under specific conditions, on a $4$-dimensional manifold $M$ together with a (very) specific Lagrangian submanifold $L$, $\Symp(M)$ is homotopy equivalent to the space of Lagrangian embeddings of $L$.

\subsubsection*{Dehn-Seidel twist and $C^0$ symplectic mapping class group}

We now turn our attention to the core of this paper. Inspired by the pioneering work of Seidel on the group of symplectomorphisms, we would like to study the topology of the group $\overline{\Symp}(M,\omega)$ of symplectic homeomorphisms. In particular, we would like to understand the $C^0$ symplectic mapping class group, i.e. the group $\pi_0(\overline{\Symp}(M,\omega)).$

There is a priori no reason for this group to be non trivial. Indeed, the flexibility results such as the $C^0$-counter example to the Arnold conjecture (\cite{BHS18}) show that sometimes symplectic homeomorphisms behave very differently than their smooth counter parts. This lead Ivan Smith to ask \footnote{in a private discussion with V. Humilière} the following question.
\begin{question}\label{question IS}
Is the square of the Dehn-Seidel twist connected to the identity in $\overline{\Symp}(M,\omega)$, where $(M,\omega)$ is a symplectic manifold as in Seidel's \autoref{thm Seidel intro}?
\end{question}

Answering this question would help to understand the relation between the symplectic mapping class group and the $C^0$ symplectic mapping class group. It would show that the natural map induced by the inclusion
\begin{equation}\label{formule J}
\pi_0(\Symp(M,\omega))\overset{J}\longrightarrow \pi_0(\overline{\Symp}(M,\omega))    
\end{equation}
is non-trivial. Here, $\overline{\Symp}(M,\omega)$, which denotes the set of symplectic homeomorphisms, is equipped with the $C^0$-topology, whereas $\Symp(M,\omega)$ is equipped with $C^\infty$-topology.

The main o ofbjective of this paper is to answer \autoref{question IS}, which is achieved by proving the following theorem.
\begin{thmx}\label{thm A}
Let $(M^{2n},\omega)$ be a $2n$-dimensionnal Liouville domain with n even and $2c_1(M,\omega)=0$. Assume that M contains an $A_2$-configuration of Lagrangian spheres $(l,l')$.

Then, for all $k\in\Z\setminus \{0\}$, $\tau_l^{2k}$ is not in the connected component of the identity in $\overline{\Symp}(M,\omega)$.
\end{thmx}
Unlike in Seidel's theorem, we only assume that $M$ contains an $A_2$-configuration. It was probably known that Seidel's \autoref{thm Seidel intro} holds for an $A_2$-configuration as well, but we were not able to find an appropriate reference.
This theorem implies that the group $\pi_0(\overline{\Symp}(M,\omega))$ is not trivial.
Of course, an immediate consequence of the previous theorem is the following corollary answering \autoref{question IS}.
\begin{corx}\label{cor thm A}
Under the same assumptions as \autoref{thm A}, the map $\tau_l^2$ is an element of infinite order in $\pi_0(\overline{\Symp}(M,\omega))$.
\end{corx}

This result comes in contrast with the work of Kauffman-Krylov \cite{KK05} and Krylov \cite{K07} which states that the Dehn-Seidel twist is of finite order in $\pi_0(\mathrm{Homeo_c}(T^*\mathbb{S}^n))$: it is $2$ when $n$ equals $2$ or $6$ and of order $4$ for all other $n$ even and of order $2$. They also proved that it is of finite order in $\pi_0(\mathrm{Diff_c}(T^*\mathbb{S}^n))$, when $n$ even, $n\neq 4$: of order $2$ when $n$ equals $2$ or $6$ and of order $4$ or $8$ for other $n$. Of course, if the Dehn-Seidel twist has order $2$ or $4$ in $\pi_0(\mathrm{Homeo_c}(T^*\mathbb{S}^n))$, then it must be of order at least than $2$ or $4$ in $\pi_0(\overline{\mathrm{Symp}}_c(T^*\mathbb{S}^n))$. \autoref{thm A} is consequently an interesting example of $C^0$ symplectic rigidity in the sense that it shows a different behaviour from what happens outside a symplectic setting: it is has infinite order, and the orders are the same for the $C^0$ context and the smooth one.

We have to discuss the relation between \autoref{thm A} and its \autoref{cor thm A}. In smooth symplectic geometry, the two results would be equivalent: indeed connected components and path connected components coincide in the $C^{\infty}$ setting. However, in $C^0$-symplectic geometry, there is no reason for this equivalence to hold and it is actually related to an important question. It is the question of the local path-connectedness of $\overline{\Ham}$, the $C^0$ closure of Hamiltonian diffeomorphisms, or $\overline{\Symp}$, which can be formulated in the following way.
\begin{question}\label{question local connected}
Given an arbitrary neighbourhood $U$ of the identity in $\overline{\Ham}(M,\omega)$ or in $\overline{\Symp}(M,\omega)$, is there a neighbourhood $V$ contained in $U$ such that every element in $V$ can be connected to the identity using a path in $V$?
\end{question}
Consequently, whether $\overline{\Symp}(M,\omega)$ is locally path-connected in dimension greater or equal to $4$ remains an open question. It is unknown whether the connected component of the identity is equal to the path-connected component of the identity in $\overline{\Symp}(M,\omega)$. 

This question is particularly complex. One could think of it as an analogue of the Nearby Lagrangian conjecture, but for symplectomorphism isotopies instead of Lagrangian isotopies. 

The nearby Lagrangian conjecture was proposed by Arnold. It states that given a cotangent bundle $T^*L$, any closed exact Lagrangian submanifold $L'\subset T^*L$ is Hamiltonian isotopic to the zero-section. This conjecture is exceptionally difficult to prove. However, important progress has been made. It was proved for $T^*\mathbb{S}^2$ by Hind \cite{H04} and $T^*\T^2$ by Goodman-Ivrii-Rizell \cite{RGI16}. On general cotangent bundles, a series of works by Fukaya-Seidel-Smith \cite{FSI07}, Abouzaid \cite{Ab10}, Kragh \cite{Kr13} and Abouzaid-Kragh \cite{AK18} led to the fact that for any closed exact Lagrangian $L'$, the projection of $L'$ onto the zero section $L$ is a homotopy equivalence.
Even if this conjecture is not proven in its full generality, those results have already been used. For instance, in Shelukhin's proof of the Viterbo conjecture \cite{Shel19}, it allows him to extend his results to all exact Lagrangian submanifolds.

\bigskip

Moreover, note that \autoref{thm A} also implies the following corollary since $\overline{\Ham}(M,\omega)$ is connected (as the closure of a connected space).
\begin{corx}\label{cor thm A Ham}
Under the hypothesis of \autoref{thm A}, $ \tau_{l}^{2k}$ does not belong to $\overline{\Ham}(M,\omega)$, for every integer $k\neq 0$.
\end{corx}

Denoting $\Symp_c(M,\omega)$ the set of compactly supported symplectomorphisms in $(M,\omega)$, the most explicit corollary may be the following one:
\begin{corx}
The Dehn-Seidel-twist along $\mathbb{S}^n$ and all its powers are in different connected components in $\overline{Symp_c}(T^*\mathbb{S}^n,\omega)$.
\end{corx}
By Weinstein's neighbourhood theorem, the $C^\infty$-counterpart of this corollary was a consequence of Seidel's \autoref{thm Seidel intro}.

In dimension $4$, me may use the nearby Lagrangian conjecture in $T^*\mathbb{S}^2$ \cite{H04} to give a simpler proof that $\tau_l$ has infinite order in $\pi_0(\overline{\mathrm{Symp}}(M,\omega))$. This is done in Appendix \ref{sec dim 4}.

\bigskip

Let us say a few words on the map $J$ defined by (\ref{formule J}). This map is very poorly understood and we have the following open question.
\begin{question}
For a general symplectic manifold $(M,\omega)$, is the map $J$ injective? Is it surjective?
\end{question}
Note that \autoref{thm A} implies that, at least, this map is non-trivial on some manifolds. Moreover, a positive answer to \autoref{question local connected} would imply the surjectivity of the map $J$.

However, in some specific cases, some results exist. As mentioned earlier in (\ref{iso MCG}), we know that, for surfaces, this map is an isomorphism.

The case of the $2n$-ball is also very interesting. Let us denote $\Symp_c(B^{2n},\omega)$ the group of compactly supported symplectomorphisms of $B^{2n}\subset\R^{2n}$. Using Alexander's trick, i.e. conjugating by $x\mapsto t\cdot x$, one gets that $\overline{\Symp}_c(B^{2n},\omega)$, the group of compactly supported symplectic homeomorphisms, is contractible. Consequently, we have that $\mathrm{MCG}^{\omega}(B^{2n},C^0)$ is trivial and so is the map $J$. On the other hand, it is not known whether the group $\Symp_c(B^{2n},\omega)$ is connected, except when $n=1\text{ or }2$. Indeed, in this case, this group is contractible. The case $n=2$ was proven by Gromov.

As this example shows, it could well be that the $C^0$ symplectic mapping class group turns out to be simpler to study in general than the smooth symplectic mapping class group.

\subsection*{Techniques involved}

Seidel's proof cannot directly be adapted to symplectic homeomorphisms. Indeed, it is based on Floer homology which only applies to smooth objects. However, we will see that barcodes form a rich enough invariant that can be defined for symplectic homeomorphisms and that offers a good substitute to Floer homology.

\subsubsection*{Floer homology}

Since the introduction of Floer Homology by Floer in \cite{Flo87}, many other Floer (co)homologies have been defined. The one we are particularly interested in is the Lagrangian intersection Floer cohomology.

We will be working with exact Lagrangian submanifolds. In an exact symplectic manifold $(M,\omega=d\lambda)$, an \emph{exact Lagrangian submanifold} $L$ is a Lagrangian submanifold such that the restriction $\lambda_{|L}$ of the $1$-form $\lambda$ is exact.

Let $L,L'$ be two closed exact Lagrangian submanifolds in an exact symplectic manifold $(M,\omega)$. We assume that their intersections are transverse. The Floer complex is generated by the intersection points $\chi(L,L')$ of the two Lagrangian submanifolds $L$ and $L'$. To define the differential, we have to count $J$-holomorphic strips, for a chosen almost complex structure $J$, between two intersection points, with boundaries on both Lagrangian submanifolds. Of course, for all the objects at stake to be well-defined, some perturbations are required. Once this Floer cohomology is defined, we have the following theorem.
\begin{thmi}[Floer \cite{Floer88}]\label{thm HF(L,L)=H(L)}
Let $(M,\omega)$ be a symplectically aspherical symplectic manifold, together with a closed weakly-exact Lagrangian submanifold $L$. Then,
$$HF^*(L,L;\Z/2)\cong H^*(L,\Z/2).$$
\end{thmi}
Many generalizations have been proved since then by Oh \cite{Oh95}, Fukaya-Oh-Ohta-Ono \cite{FOOO09}...

One of the many interesting properties of this cohomology is its Hamiltonian invariance, i.e. let $\phi$ be a Hamiltonian diffeomorphism on $M$, then
$$HF^*(L,L')\cong HF^*(L,\phi(L')).$$
When $L=L'$, we denote $HF(L,H)=HF(L,\varphi_H^1(L))$ and for all $K,H\in\Ham(M,\omega)$, we have $HF(L,H)\cong HF(L,K)$.
The invariance property makes this cohomology a great tool to study Hamiltonian diffeomorphisms.

Moreover, the structure of this cohomology is very rich. Indeed, given three closed exact Lagrangian submanifolds $L_0,L_1,L_2$ in $(M,\omega)$, counting pseudo-holomorphic curves between three intersection points, one can define a product structure
$$\mu^2:HF(L_0,L_1)\otimes HF(L_2,L_0)\rightarrow HF(L_2,L_1).$$
This product equips $HF(L,L)$ with a ring structure and the isomorphism of \autoref{thm HF(L,L)=H(L)} is a ring isomorphism.
Given more Lagrangian submanifolds, we can also define higher products $\mu^k, k\in\N$.

\subsubsection*{Action selectors and Barcodes}

Action selectors were introduced by Viterbo \cite{Vit92} for Lagrangian submanifolds in a cotangent bundle using generating functions theory. After this construction, it was adapted to many contexts by Oh \cite{Oh04}, Schwarz \cite{Sch00}, Leclercq \cite{Lec07} and others... They contributed to the definition of many useful tools, such as the spectral norm \cite{Vit92}, or the study of other ones such as the Hofer norm, defined for Hamiltonian diffeomorphisms \cite{Hof90}. These action selectors are fundamental symplectic invariants and are thus deeply studied. Since these objects will be discussed in much more detail later on, we will be brief here.

Given a non-zero homology class $\alpha\in HF(H)$ (respectively a cohomology class in $HF(L,H)$), the associated action selector $l(\alpha,H)$ is the minimal action above (respectively maximal action under) which this class is represented in homology.
These action selectors have been subject to a lot of works and have been shown to satisfy many interesting properties. For instance, one relevant result for us is that Buhovsky-Humilière-Seyfaddini \cite{BHS19} proved that they are locally $C^0$-Lipschitz in the Hamiltonian Floer homology case.

Thanks to \autoref{thm HF(L,L)=H(L)}, one can define the spectral norm $\gamma$ by
$$\gamma(L,H)=l([L],H)-l([pt],H).$$ 

This spectral norm is continuous with respect to a certain distance, the Hofer distance.

\bigskip

Barcodes come from a totally different area of mathematics: topological data analysis. A barcode is a collection of intervals (called bars) used to represent certain algebraic structures called persistence modules. They were introduced by Edelsbrunner et al. \cite{Edel&al00} and, for example, found applications in image recognition with the work of Carlson et al. \cite{Carlson&al04}.

The terminology of barcodes was brought into symplectic topology by Polterovich and Shelukhin \cite{PolShel14} although germs of this theory were already present in the work of Barannikov \cite{Bar94} and Usher \cite{Ush11,Ush14}. Indeed, they observed that Floer theories carry natural persistence module structures, coming from the action filtration.

The space of barcodes may be equipped with a distance, called the \emph{bottleneck distance}. One can associate a barcode to a Morse function, and this barcode is $C^0$-continuous with respect to the Morse function. They satisfy many more properties that will be discussed in much more details later.

Barcodes are of particular interest since they carry the information on the action filtration in Floer (co)homology. Given two exact Lagrangian submanifolds $L,L'$ in a symplectic manifold $(M,\omega)$, this filtration is given by the cohomology of the following subcomplexes. For all $\kappa\in\R$, we define
$$CF^{*,\kappa}(L,L')=\mathrm{span}_{\Z/2}\left\{z\in\chi(L,L'),~~\calA_{L,L'}(z)<\kappa\right\}\subset CF^*(L,L'),$$
where $\chi(L,L')$ denotes the generators of the Floer complex $CF(L,L')$, and $\calA_{L,L'}$ the action functional associated to the pair of Lagrangians. When the parameter $\kappa$ increases, some classes appear while some other ones vanish. The bars of the associated barcode encode the levels at which classes appear and disappear.

But maybe the most interesting property of the space of barcodes is that, being equipped with a distance, it has a topology. This allows us to state our following main tool-theorem, giving a local $C^0$-Lipschitz continuity of the barcodes.
\begin{thmx}\label{thmx loc lip}
Let $M$ be a Liouville domain. Let $L$ and $L'$ be two closed exact Lagrangian submanifolds, and assume that $H^1(L',\R)=0$. The map
$$\varphi\in\Symp(M,\omega)\mapsto \hat{B}(\varphi(L'),L),$$
where $\hat{B}(L,L')$ denotes the barcodes associated to the exact Lagrangian submanifolds $L$ and $L'$, is locally Lipschitz continuous with respect to the $C^0$-distance and extends continuously to $\overline{\Symp}(M,\omega)$.
\end{thmx}

To prove this theorem, we will adapt the proof of some recent continuity results to our context. The first useful result comes from a work of Kislev-Shelukhin. In \cite{KS18}, they proved that, in the case of a Lagrangian submanifold together with a Hamiltonian function, the aforementioned barcodes are continuous with respect to the Lagrangian spectral norm $\gamma(L,H)$.

The second result is the one we mentioned before: Buhovsky-Humilière-Seyfaddini \cite{BHS19} proved that action selectors (in the Hamiltonian case) are locally $C^0$-Lipschitz. This allows to extend these objects and the different spectral invariants to the $C^0$-closure, i.e. to Hamiltonian homeomorphisms. This provides invariants that will be used to study these objects.

\subsection*{Organisation}

In the first section, we  give some notations and conventions that will be used in the rest of this paper. The second section is a short presentation of the theory of persistence modules and barcodes focusing on the properties we are interested in. We also prove some small topological observations on this set, both completeness and connectivity results. We then give the definition the barcodes for Lagrangian Floer cohomology. We then prove that the product operations in Floer cohomology respects the filtration. We also present the action selectors for a pair of Lagrangian and define the spectral distance in the case of a pair exact Lagrangians non-necessarily Hamiltonian isotopic, together with some properties. Note that the same definition also appears in Shelukhin's work \cite{Shel19}.

The third section is the proof of our main tool-theorem used to get our results on the Dehn-Seidel twist. We prove \autoref{thmx loc lip}. Using this theorem, we also prove the two points of \autoref{thm barcodes connected}. The first point is a connectivity result while the second one  associates a continuous path of barcodes to a continuous path in $\overline{\Symp}$.

Finally, in the last section, we state and prove our main results, \autoref{thm A} and its corollaries.

We present proofs, without the use of barcode, of weaker results for dimension $4$ in Appendix \ref{sec dim 4}.

\subsection*{Acknowledgments}

This paper results from the author's PhD thesis at the Institut Mathématiques de Jussieu-Paris Rive Gauche (IMJ-PRG) and he is utterly grateful to his wonderful PhD advisors, Vincent Humilière and Alexandru Oancea, for all their support and mentoring. He is thankful for the partial support by the ANR project ``Microlocal" ANR-15-CE40-0007. It was finished during a post-doc position supported by the FNS and so the author thanks Université de Neuchâtel for its hospitality. The author also thanks Sobhan Seyfaddini and Felix Schlenk for beautiful insights, Michael Usher and Jean-François Barraud for their so precious advice and Côme Dattin for many hours of deeply profitable conversations. The author is deeply thankful to Ailsa Keating for useful discussions and suggestions, in particular for \autoref{prop rk cohom Dehn twist}.
\section{Preliminaries}\label{sec preli}
All the following notions in this section are originally due to Floer \cite{Floer88}. One can also refer to e.g. Auroux \cite{Aur14}, Oh \cite{Oh95}, Seidel \cite{Sei08}...

Let $(M,\omega)$ be a Liouville domain, with $d\lambda=\omega$, and let $L$ and $L'$ be two closed connected exact Lagrangian submanifolds in $M$. We denote $f_L:L\rightarrow \R$ and $f_{L'}:L'\rightarrow \R$ the functions satisfying $df_L=\lambda_{|L}$ and $df_{L'}=\lambda_{|L'}$. We recall that these functions are well-defined up to a constant.
\begin{dfn}\label{def action} In our context, the \emph{action functional} on the space of paths from $L$ to $L'$ $\mathcal{P}(L,L')$ is the map $\calA_{L,L'}:\mathcal{P}(L,L')\rightarrow\R$ defined by the expression
$$\calA_{L,L'}(\gamma)=\int \gamma^*\lambda + f_L(\gamma(0))-f_{L'}(\gamma(1)),$$
with $\gamma\in \mathcal{P}(L,L')$.
\end{dfn}

\begin{rk}
This definition of the action presents an unusual choice regarding the classical conventions used in cohomology. Indeed the differential in cohomology decreases this action. This choice does not fundamentally matter but it makes the definitions of persistence modules and barcodes easier as our setting thus matches with the usual definitions of these objects.
\end{rk}

The critical points of $\calA_{L,L'}$ are the intersection points between $L$ and $L'$. At such a point $p$, we have
$$\calA_{L,L'}(p)=f_L(p)-f_{L'}(p).$$
We denote $\Spec(L,L')$ the set of critical values of $\calA_{L,L'}$ and $\chi(L,L')$, the intersection points between $L$ and $L'$. These are the generators of the chain complex of Floer cohomology.
We can associate an action to a set of intersection points. Let $p_1,...p_k$, for $k\in\N$ be points in $\chi(L,L')$. The action of the formal sum of these points is the maximum of the different actions, i.e.
$$\calA_{L,L'}(p_1+...+p_k)=\max\{\calA_{L,L'}(p_1),...,\calA_{L,L'}(p_k)\}.$$

Since the energy of a Floer strip connecting $p$ to $q$ is always strictly positive, the differential strictly decreases the action, i.e. 
$$\calA_{L,L'}(p)>\calA_{L,L'}(\partial p),$$
for all $p$ in $\chi(L,L')$, with $\partial$ denoting the Floer differential.

To achieve the transversality and compactness of the moduli spaces as well as the transversality of the intersections required to define Floer cohomology, we need to consider Hamiltonian and (time dependant) almost-complex structure perturbations, which we will denote by the pair $(H,J_t)$ or simply $(H,J)$. The generators of the Floer complex are then the flow lines $\gamma:[0;1]\rightarrow M$ such that
$$\dot{\gamma}(t)=X_H(t,\gamma(t)),$$
$$\gamma(0)\in L,\gamma(1)\in L'.$$
We will denote $\chi_H(L,L')$ these generators of the Floer complex. When we define the action in this context, we have to take into account the Hamiltonian perturbation. The Hamiltonian action of a path $\gamma$ from $L$ to $L'$ is then defined as
\begin{equation}\label{formule action hamiltonienne}
    \calA_{L,L'}^H(\gamma)=\int_0^1\gamma^*\lambda-H(\gamma)dt+f_L(\gamma(0))-f_{L'}(\gamma(1)).
\end{equation}
We denote by $\Spec(L,L';H)$ the set of critical values of this action functional. The critical points are the above mentioned generators of the Floer complex. We now get the Floer complex
$$CF(L,L';H,J),$$
and then the Floer cohomology $HF(L,L';H,J)$.

\begin{rk}\label{rk perturbation petite}
If the Lagrangian submanifolds $L$ and $L'$ were transverse, we could of course choose $H=0$. Moreover, for two given Lagrangian submanifolds, the Hamiltonian perturbation to achieve transversality can be chosen as small as desired.
\end{rk}

\begin{rk}\label{rk eg action intersection}
The generators can be seen as intersection points between $\phi_H^1(L)$ and $L'$ and so we have an identification of the Floer complexes $CF(L,L';H)$ and $CF(\phi_H^1(L),L';0)$. Using for example Proposition 9.3.1 in \cite{MS3rd}, one can show that up to a constant, the action defined with Equation (\ref{formule action hamiltonienne}) corresponds to $f_{\phi_H^1(L)}-f_{L'}$.
\end{rk}

In the following Sections, we will either consider the Lagrangian submanifolds $L$ and $L'$ as Lagrangian submanifolds in $M$ or as Lagrangian submanifolds in $T^*L$. We will denote $HF(L,L';H,J,M)$ when the Floer cohomology is computed in $M$ and $HF(L,L';H,J,T^*L)$ when the Floer cohomology is computed in $T^*L$.

\bigskip

Let $(M,\omega)$ and $(M',\omega')$ be two Liouville domains, together with two pairs of closed exact Lagrangian submanifolds $(L_0,L_1)\subset M$ and $(L_0',L_1')\subset M'$. Let us recall that we are working with $\Z/2$-coefficients. Then, there is a Künneth-type formula
\begin{equation}\label{formule def Kunneth}
HF(L_0,L_1;H,J)\otimes HF(L_0',L_1';H',J')\cong HF(L_0\times L_0',L_1\times L_1';H\oplus H',J\oplus J').    
\end{equation}
This isomorphism is natural, resulting from the fact that a pseudo holomorphic curve $v$ in $(M\times M',J\oplus J')$ can be written as $v=(u,u')$, where $u$ is pseudo-holomorphic curve in $M$ and $u'$ in $M'$. At the chain level, for $(p,p')\in\chi_H(L_0,L_1)\times \chi_{H'}(L_0',L_1')$, the isomorphism is simply defined by
$$(p,p')\mapsto (p,p')\in\chi(L_0\times L_0',L_1\times L_1').$$

\bigskip

To choose some conventions, we point out that the case when $L$ and $L'$ and actually assume that $L'=L$. This choice is not restrictive thanks to the Hamiltonian invariance of the Floer cohomology.

In this case, it is indeed easier to work with more general conditions on the Lagrangian submanifolds considered. Due to Weinstein's neighbourhood theorem and energy estimates, choosing to work in the cotangent bundle $T^*L$ of the Lagrangian $L$ will not be restrictive. A longer and more detailed discussion on this subject will be held in Section \ref{sec barcode cotangent}.

Let $\varepsilon>0$ and choose a $\varepsilon$-small Morse function $f:L\rightarrow\R$. We extend this function to $T^*L$ by setting
\begin{equation}\label{formule Ham pour Morse}
    H=f\circ\pi:T^*L\rightarrow\R,
\end{equation}
where $\pi:T^*L\rightarrow L$ is the natural projection. The exact Lagrangian submanifold $\phi_H(L)$ is the graph of $df$ and intersects $L$ transversely. Note that if we work in a symplectic manifold $M$ instead of $T^*L$, the cotangent bundle of $L$, we have to multiply $H$ by a cut-off function equal to $1$ near $L$.

With this perturbation, a critical point $p$ of $f$ is exactly an intersection point between $L$ and $\phi_H(L)$. We then obtain
\begin{equation}\label{formule action Morse Floer}
    \calA^H_{L,\phi_H(L)}(p)=-H(p).
\end{equation}

For a good choice of almost-complex structure $J$ and of shift in the definition of the degree of the intersection points, the matching associates a generator of the Floer cochain complex $CF(L,L;H,J)$ of degree $i$ to a critical point of Morse index $n-i$, i.e a generator of the Morse cochain complex $CM(L,H)$ of index $i$ \cite{Flo89}.

This identification is associated to a correspondence between the moduli spaces. The Floer cochain complex $CF(L,L;H,J)$ is then identified with the Morse cochain complex $CM(L,H)$. Together with the Hamiltonian invariance of Floer cohomology, it implies the following proposition.
\begin{prop}[Floer \cite{Floer88}]\label{prop lien entre HF et H}
Let $L$ and $L'$ be two Lagrangian submanifolds which are Hamiltonian isotopic to each other, such that $[\omega]\cdot \pi_2(M,L)=[\omega]\cdot \pi_2(M,L')=0$, then
$$HF^*(L,L')\cong HF(L,L)\cong H^*(L;\Z/2).$$
\end{prop}
We assume in this statement that the choice of shift in the definition of the degree for the generators of the Floer complexes make the degree equal to the Morse index.

\begin{rk}\label{rk filtration HF et HM}
Both the Floer cochain complex and the Morse cochain complex carry a natural filtration that will be discussed in details in Section \ref{chap barcode Floer}. The filtration for the Floer complex is given by the action functional. The filtration for the Morse complex is given by the Morse function $f$.

However, with our choice of action for Floer cohomology, the identification between these two complexes does not respect these natural filtrations. Indeed the differential decreases the action functional in Floer cohomology while the differential increases the action in Morse cohomology. Consequently we have to consider the filtration given by $-f$. We denote $CF(L,L;H,J;\calA^H_{L,L})$ the Floer cochain complex with the filtration given by $\calA^H_{L,L}$ and $CM(L,f;-f)$ the Morse cochain complex with the filtration given by $-f$.

Together with the formula (\ref{formule action Morse Floer}), this leads, for the $\varepsilon$-small Hamiltonian defined in the formula \ref{formule Ham pour Morse}, to
$$CF(L,L;H,J;\calA^H_{L,L})\cong CM(L,H;-H).$$
\end{rk}

\begin{rk}\label{rk representant classe fondamentale unique}
We can choose the Morse function $f$ to have a unique maximum and a unique minimum on $L$. This implies that there is a unique generator of $CM^0(L,H)$ and a unique generator of $CM^n(L,H)$. With the previously mentioned good choice of grading, this implies that there is also a unique generator of $CF^0(L,L;H,J)$ and a unique generator of $CF^n(L,L;H,J)$.
\end{rk}
\bigskip

In the following sections, we will not be interested in the Hamiltonian or almost-complex structure perturbation, we will just want these Hamiltonian perturbations to be $\varepsilon$-small, for a given $\varepsilon>0$. Thus, using Kislev-Shelukhin's notations \cite{KS18}, we will denote the Floer complex of $L$ and $L'$
$$CF^*(L,L';\mathcal{D}),$$
where $\mathcal{D}$ denotes the data perturbation, i.e. the pair $(H,J)$. The set of generators will then be denoted $\chi_{\mathcal{D}}(L,L')$. The perturbation data is said to be $\varepsilon$-small if the Hamiltonian is $\varepsilon$-small. When not needed, we will just write $CF^*(L,L')$, and assume that there is a suitable perturbation data implied.

\bigskip

We now have to make some remarks concerning the relation between action and energy, when there is a data perturbation $\calD$. As we will later only be concerned about $C^2$-small perturbations, we will only describe this situation here. However, if one wants to compute Floer cohomology for a particular Hamiltonian $H$, this Hamiltonian term has to be taken into account when defining the action of the generators of the Floer complex. We can choose a perturbation data to achieve transversality everywhere and conduct the same argument as the following.

Let $p,q$ be two perturbed intersection points in $\chi_{\mathcal{D}}(L,L')$ together with $u$, a $J$-holomorphic strip from $p$ to $q$. When computing the energy $E(u)$, one has to take into account the perturbation data.
So the energy writes as
$$E(u)=\calA_{L,L'}(p)-\calA_{L,L'}(q)+ f_{\calD}(p,q),$$
where $f_{\calD}$ is a function depending smoothly on $\calD$ and such that $f_{\calD}$ converges to zero when the Hamiltonian part of the perturbation data $\calD$ goes to zero.
According to \autoref{rk perturbation petite}, this perturbation data can be chosen as small as wished, so that, for all $\varepsilon>0$, we can find $\calD$ such that
\begin{equation}\label{energie action avec perturbation}
    E(u)\leq \calA_{L,L'}(p)-\calA_{L,L'}(q) + \varepsilon,
\end{equation}
and thus $$\calA_{L,L'}(q)\leq\calA_{L,L'}(p)+\varepsilon.$$
This last remark will one of the key arguments in Section \ref{sec Floer barcode} to define persistence modules and barcodes associated to Lagrangian Floer cohomology.

\bigskip

The Floer cochain complex can be equipped with product operations. We will only give the basic ideas, for more details see e.g. Auroux presentation in \cite{Aur14} or the books of Oh \cite{Oh15} and Seidel \cite{Sei08}.

Let $L_0,L_1$ and $L_2$ be three Lagrangian submanifolds of a symplectic manifold $(M,\omega)$. Under suitable assumptions, we define a product operation from the Floer complexes $CF(L_1,L_2,\calD)$ and $CF(L_0,L_1,\calD')$ to $CF(L_0,L_2,\calD'')$ for a suitable choice of perturbation data collection, i.e. a linear map
$$ CF(L_1,L_2,\calD) \otimes CF(L_0,L_1,\calD) \rightarrow CF(L_0,L_2,\calD),$$
which induces a well-defined product
$$HF(L_1,L_2,\calD) \otimes HF(L_0,L_1,\calD) \rightarrow HF(L_0,L_2,\calD).$$
We recall that we can in fact define, given $k+1$ exact Lagrangian submanifolds $L_0,...,L_k$ in a Liouville domain $(M,\omega)$, a map
$$\mu^k:CF(L_{k-1},L_k,\calD)\otimes\cdots\otimes CF(L_0,L_1,\calD)\rightarrow CF(L_0,L_k,\calD).$$
This map is $(2-k)$-graded when it is possible to define a grading on the Floer complexes.

Moreover, we have the following property \cite{Sei08}.
\begin{prop}
Let $L$ and $L'$ be two closed exact Lagrangian submanifolds in $M$. The product 
$$CF(L',L,\calD)\otimes CF(L',L',\calD)\rightarrow CF(L',L,\calD)$$
is cohomologically unital. This unit is given by the image of the fundamental class $[L']$ of $L'$ in $HF(L',L')$.
\end{prop}
Adequate choices of Hamiltonian perturbations together with Morse-Bott theory make it possible to have an isomorphism on the level of cochain complexes which leads to the following proposition (see \cite{BC07}).
\begin{prop}\label{prop identite mult z}
Let $L$ and $L'$ be two closed exact Lagrangian submanifolds in $M$  and $\varepsilon>0$. Let $f$ a Hamiltonian perturbation for $(L',L')$ defined as in \autoref{rk representant classe fondamentale unique} and $H$ a Hamiltonian perturbation for $(L',L)$. Assume that $f$ and $H$ are $\varepsilon$-small. Then for $\varepsilon$ small enough, the following map is an isomorphism: 
$$\mu^2(\cdot,z):CF(L',L;H)\rightarrow CF(L',L;H_f),$$
where $z$ is the unique representative of the image (\autoref{prop lien entre HF et H}) of the fundamental class $[L']$ in $CF(L',L';f)$ and $H_f=f\sharp H$, the perturbation of $H$ by the Hamiltonian $f$.
\end{prop}
\section{Barcodes and action selectors in symplectic topology}\label{chap barcode Floer}
\subsection{Persistence modules and barcodes}\label{sec pers mod barcodes}

Persistence module over a field $\K$ is a family $(V^t)_{t\in\R}$ of finite dimensional vector spaces over $\K$ equipped with a doubly-indexed family of linear maps, called structure maps, $i_t^s:V^s\rightarrow V^t$, for all $s\leq t\in \R$ satisfying:
\begin{enumerate}
    \item $V^t=0$ for $t\ll0$,
    \item for all $s,t,r\in\R$, such that $r\leq s\leq t$, we have $i_t^s\circ i_s^r=i^r_t$ and $i_s^s=\Id_{V^s}$,
    \item for all $r\in\R$, there is $\varepsilon>0$ such that $i^s_t$ are isomorphisms for all $r-\varepsilon<s\leq t \leq r$,
    \item there is a finite set of points $S(V)\subset\R$ such that for all $r\in\R\setminus S(V)$, there exists $\varepsilon>0$ such that $i^s_t$ are isomorphisms for all $r-\varepsilon<s\leq t<r+\varepsilon$.
    \end{enumerate}
We will denote the persistence module $V$ or $(V,i)$. The set $S(V)$ is called the spectrum of $V$. We will denote by $V^{\infty}$ the direct limit
$V^{\infty}=\varinjlim\limits_{ t\to + \infty}V^t,$ together with $i^s:V^s\rightarrow V^{\infty}$ the natural map.

Let $(V,i)$ and $(V',i')$ be two persistence modules. A morphism of persistence modules $h:(V,i)\rightarrow (V',i')$ is a family of morphisms $h^t:V^t\rightarrow V'^t, t\in\R$ such that $h^ti^s_t=i'^s_th^s$ for $s<t$.

Let $(V,i)$ be a persistence module, and $\delta\geq 0$. The $\delta$-shifted persistence module $(V[\delta],i[\delta])$ is the persistence module with vector spaces $V[\delta]^t=V^{t+\delta}$ and maps $i[\delta]_t^s=i^{s+\delta}_{t+\delta}$. We will denote $sh(\delta)_V:V\rightarrow V[\delta]$ the natural shift morphism of persistence modules given by
$$sh(\delta)^t_V= i^{t}_{t+\delta}:V^t\rightarrow V^{t+\delta}.$$
A morphism of persistence modules $h:V\rightarrow V'$ naturally induces a shifted morphism of shifted persistence modules $h[\delta]: V[\delta]\rightarrow V'[\delta]$.
For $\delta\leq 0$, we denote $V[\delta]$ the persistence module such that $V[\delta][-\delta]\cong V$.

Given $V$ and $V'$ be two persistence modules together with $\delta,\varepsilon\geq0$, we say that they are \emph{$(\delta,\varepsilon)$-interleaved} if there exist two morphisms of persistence modules $f:V\rightarrow V'[\delta]$ and $g:V'\rightarrow V'[\varepsilon]$ such that 
$$g[\delta]\circ f = sh(\delta+\varepsilon)_V \text{ and } f[\varepsilon]\circ g = sh(\delta+\varepsilon)_{V'}.$$
The pair $(f,g)$ is called a $(\delta,\varepsilon)$-interleaving. If $\varepsilon=\delta$, it is a \emph{$\delta$-interleaving}, and $V,V'$ are $\delta$-interleaved. The interleaving distance between $V$ and $V'$ is then
$$d_{inter}(V,V')=\inf\{\delta ~~| ~~ V, V'\text{ are }\delta\text{-interleaved}\}.$$
This interleaving distance satisfies the triangle inequality. Let $U,V$ and $W$ be three persistence modules, then
\begin{equation}\label{tri ineq}
    d_{inter}(U,W)\leq d_{inter}(U,V) + d_{inter}(V,W).
\end{equation}

Let us now introduce the closely related notion of barcodes.

Let $J$ be a non-empty interval in $\R$ of the form $(a,b]$ or $(a,+\infty)$, with $a$ and $b$ in $\R$. The \emph{interval module} $I=\K^J$ is the persistence module with vector spaces
\begin{align*}
  I^t =
  \begin{cases}
    \K, \text{ if }  t\in J\\
    0, \text{ otherwise,}
  \end{cases}
\end{align*}
and structure maps
\begin{align*}
  i^s_t =
  \begin{cases}
    \Id, \text{ if }  s,t \in J,\\
    0, \text{ otherwise.}
  \end{cases}
\end{align*}

The following structure theorem, proven in \cite{CB15}, relates barcodes and persistence modules.
\begin{thm}
For any persistence module $V$, there is a unique collection of pairwise distinct intervals $(J_i)_{i\in \mathcal{I}}$ of the form $(a_i,b_i]$ or $(a_i,+\infty)$, with $a_i,b_i\in S(V)$, and multiplicity $m_i\in\N$ such that
$$V\cong \bigoplus_{i\in \mathcal{I}} (\K^{J_i})^{m_i}.$$
\end{thm}

From this theorem, we can associate a barcode associated to $V$.
A \emph{multiset} is a pair $B=(S,m)$ where $S$ is a set and $m:S\rightarrow \N\cup\{+\infty\}$ is the multiplicity function. This function tells how many times each $s\in S$ occurs in $B$.

\begin{dfn}
We denote by $B(V)$ the multiset containing $m_J$ copies of each interval $J$ appearing in the structure theorem, and $\calI(B(V))$ the set of intervals $J_i$ without multiplicity. $B(V)$ is called the barcode associated to $V$, and the intervals $J_i$ are called bars. We will denote
$$B(V)=\bigoplus_{J\in\calI(B(V))}J^{m_J}.$$
\end{dfn}

We can equip the set of barcodes with a distance, which is called the bottleneck distance.
\begin{dfn}
Let $I$ be a non-empty interval of the form $(a,b]$ or $(a,+\infty)$, and $\delta\in\R$ such that $2\delta<b-a$. We denote $I^{-\delta}$ the interval $(a-\delta,b+\delta]$ or $(a-\delta,+\infty)$. Let $B$ and $B'$ be two barcodes, and $\delta\geq 0$. They admit a $\delta$-matching if we can delete in both of them some bars of length smaller than $2\delta$ to get two barcodes $\bar{B}$ and $\bar{B}'$ and find a bijection $\phi:\bar{B}\rightarrow \bar{B}'$ such that if $\phi(I)=J$, then
$$I\subset J^{-\delta} \text{ and } J\subset I^{-\delta}.$$
\end{dfn}
As it was the case for persistence modules, the bottleneck-distance between the barcodes $B$ and $B'$ is then defined as
$$d_{bottle}(B,B')= \inf\{\delta |~~B \text{ and }B' \text{ admit a $\delta$-matching}\}.$$
The bottleneck distance is non-degenerate: if $B$ and $B'$ are two barcodes such that $d_{bottle}(B,B')=0$, then $B=B'$.

The two notions of interleaving and bottleneck distance are closely related: an \emph{isometry theorem} \cite{BL14} states that for $V,V'$ two persistence modules,
$$d_{inter}(V,V')=d_{bottle}(B(V),B(V')).$$
As for persistence modules, given a barcode $B$ and $\delta \in\R$, we will denote $B[\delta]$ the barcode obtained from $B$ by an overall shift of $\delta$. If $B$ is a barcode associated to a persistence module $V$, then $B[\delta]$ is the barcode associated with the persistence module $V[\delta]$.

\subsection{A bit of topology}\label{sec topology barcodes}

One of the main objectives of this work is to obtain new information concerning $C^0$-symplectic topology using the technology of barcodes applied to Floer homology. We will be working on cases where the number of generators of the chain complex is finite.

\begin{dfn}
A barcode is said to be \emph{finite} if it contains finitely many intervals counted with multiplicity. We will denote $\calB_f$ the set of finite barcodes.
\end{dfn}

Since we study $C^0$ objects in a world of smoothness, we need, at some point, to take limits, and hence limits of finite barcodes which are not necessarily finite. Thus, the question of closedness and completeness naturally arise.
This will be achieved through the set-up given in the following definition.

\begin{dfn}\label{def Q-tame barcodes}
We denote by $\calB$ the set of barcodes satisfying the following condition: for all $\varepsilon>0$, the number of bars of length greater or equal to $\varepsilon$ is finite.
\end{dfn}
\begin{rk}
In \cite{Chazal&al16,BV18} such barcodes are referred to as ``$q$-tame barcodes".
\end{rk}

The following proposition, proved by Bubenik and Vergili \cite{BV18}, justifies the introduction of the set $\calB$.
\begin{prop}\label{prop barcode complet}
The space $\calB$ is complete.
\end{prop}
Aside from completeness, the following proposition explains why the set $\calB$ is of particular interest for us, as we will be working with limits of sequences of finite barcodes.
\begin{prop}
The set $\calB_f$ is dense in $\calB$ for the topology induced by the bottleneck distance.
\end{prop}
\begin{proof}
Let us pick $B\in\calB$. We set $(B_n)_{n\in\N}$ a sequence of barcodes defined by
$$B_n=\bigoplus\limits_{\substack{I\in \calI(B) \\ l(I)\geq \frac{1}{n}}}I^{m_I},$$
where $l(I)$ is the length of the interval $I$, and $m_I$ its multiplicity in $B$. By definition of $\calB$, for all $n\in\N$, $B_n$ is a finite barcode, and for all $n\in\N$, $B_n$ satisfies
$$d_{bottle}(B_n,B)=\frac{1}{2n}.$$
This implies that $(B_n)_{n\in\N}\subset\calB_f$ converges to $B$ for the bottleneck distance.
\end{proof}

When studying homology or cohomology, the presence of a $\Z$-grading is important. We can easily incorporate this notion to obtain those of persistence modules of $\Z$-graded vector spaces such that the structure maps respect the grading. For instance, if we have a family of persistence modules $V_r$ indexed by the integers, the persistence module
$$(\oplus_{r\in\Z} V_r,\oplus_{r\in\Z} i_r)$$
has such a structure. We can then define an interleaving distance as
$$d_{inter}(V,V')=\max\limits_{r\in\Z}\{d_{inter}(V_r,V_r')\},$$
where $V=\oplus_r V_r$ and $V'=\oplus_r V'_r$ are two $\Z$-graded persistence modules.

We can incorporate this notion in the same way for barcodes. A $\Z$-graded barcode is a family of barcodes $(B_r)_{r\in\Z}$. We denote
$$B=\bigoplus\limits_{r\in \Z}B_r,$$
the $\Z$-graded barcode $B$ associated to the family $(B_r)_{r\in\Z}$.
Then, as for persistence modules, the bottleneck distance for graded barcodes is defined by
$$d_{bottle}(B,B')=\max_{r\in\Z}\{d_{bottle}(B_r,B'_r)\},$$
where $B=\bigoplus\limits_{r\in \Z}B_r$ and $B'=\bigoplus\limits_{r\in \Z}B'_r$.

\begin{rk}
Let $B=\bigoplus\limits_{r\in \Z}B_r$ be a $\Z$-graded barcode and let $I$ be a bar in $B$. We call \emph{index} of $I$, denoted $\mathrm{Ind}(I)$, the integer $r\in\Z$ such that $I$ is a bar of $B_r$.
\end{rk}

A finite graded barcode $B=(B_r)_{r\in\Z}$ is a graded barcode such that there is finitely many bars in the whole family $(B_r)_{r\in\Z}$. Since we will always consider graded barcodes, we also denote $\calB_f$ the set of finite graded barcodes. By abuse of notations, we also denote $\calB$ the set of graded barcodes $B=(B_r)_{r\in\Z}$ such that for all $\varepsilon>0$ there is a finite number of bars of length greater than $\varepsilon$ in the whole family $(B_r)_{r\in\Z}$. From now on, when using the notation $\calB$ or $\calB_f$, we will always refer to their graded version.

\begin{rk}
Let $B=(B_r)_{r\in\Z}$, if $B$ is finite, then $B_r$ has more than $0$ bars for only finitely many $r\in\Z$. In the same way, if $B\in\calB$, then for all $\varepsilon>0$, $B_r$ has more than $0$ bars of length greater than $0$ for only finitely many $r\in\Z$.
\end{rk}

With these graded barcodes, we still have
$$\overline{\calB_f}=\calB,$$
and for the same reason as in the non-graded case, $\calB$ is complete.
\bigskip

Before moving on and defining barcodes for objects of real interest, we have to make some observations regarding the connectedness of $\calB$.

First of all, let us introduce the map that counts the number of semi-infinite bars in each degree.
\begin{dfn}\label{def compte de barres semi-infinies}
We define $\sigma^\infty:\calB\rightarrow\N^\Z$ by 
$$\sigma^\infty(B)=(\sigma^n)_{n\in \Z}\text{ with } \forall n\in\Z,~~\sigma^n=\sum\limits_{\substack{I\in \calI(B) \\ l(I)=+\infty \\ \mathrm{Ind}(I)=n}}m_I.$$
\end{dfn}
This map will be very useful. Indeed the following property shows that its relation with the bottleneck distance is quite straightforward.

\begin{prop}
For all $B,B'\in \calB$, $$d_{bottle}(B,B')<+\infty \iff \sigma^\infty(B)=\sigma^\infty(B').$$
\end{prop}
Since the proof is straight-forward, we leave it to the reader to check that this lemma indeed holds. For a complete proof, one can refer to the author's thesis.

\begin{rk}
The barcode $B^\infty$ introduced in the proof strongly relates to what we defined as $V^\infty$ at the beginning of Subsection \ref{sec pers mod barcodes}. The number of bars in each degree $r\in\Z$ is equal to the dimension of the degree $r$ component of $V^\infty$.
\end{rk}
This proposition immediately implies the following corollary, which is topologically really useful.
\begin{cor}
$\sigma^\infty$ is locally constant.
\end{cor}

Thanks to this corollary and \autoref{def compte de barres semi-infinies} of $\sigma^\infty$, we can now state the following proposition.
\begin{prop}\label{prop connected components barcodes}
The connected components of $\calB$ are indexed by the graded number of semi-infinite bars, i.e. two barcodes belong to the same connected component of $\calB$ if and only if they have the same number of semi-infinite bars in each degree. Moreover the connected components are path-connected.
\end{prop}

\begin{proof}
Since the map $\sigma^\infty$ is locally constant, it is constant on the connected components of $\calB$. This means that if two barcodes $B,C\in\calB$ are in the same connected component, then $\sigma^\infty(B)=\sigma^\infty(C)$, i.e. $B$ and $C$ have the same number of semi-infinite bars in each degree.

Conversely, let $B$ be a barcode in $\calB$. With $r\in\Z$ denoting the degree, we write $B=\bigoplus_{r\in\Z}B^r$ and denote
$$B^r=\bigoplus_{i\in\mathcal{I}_B^r}(a_i,+\infty)\oplus\bigoplus_{i\in\mathcal{J}_r}(a_j,b_j].$$
We define for all $t\in[0,1]$
$$B^r_t=\bigoplus_{i\in\mathcal{I}_B^r}((1-t)a_i,+\infty)\oplus\bigoplus_{i\in\mathcal{J}_B^r}((1-t)a_j,(1-t)b_j],$$
and $B_t=\bigoplus_{r\in\Z}B^r_t$. The path $(B_t)_{t\in[0,1]}$ is a continuous path of barcodes from $B$ to $$B_0(B)=\bigoplus_{r\in\Z}\bigoplus_{i\in\mathcal{I}_B^r}(0,+\infty).$$
Let $B$ and $C$ be two barcodes in $\calB$ such that they have the same number of semi-infinite bars in each degree. Then for all $r\in\Z$, $\calI_B^r=\calI_C^r$ so $B_0(B)=B_0(C)$.

This implies that the two barcodes $B$ and $C$ are isotopic and thus in the same connected component of $\calB$ which concludes the proof of this proposition.
\end{proof}

The following corollary is a direct and obvious consequence of the previous \autoref{prop connected components barcodes}, but its formulation will be useful later.

\begin{cor}\label{prop path inf bar}
Let $(B^t)_{t\in[0;1]}$ be a continuous path of graded barcodes. Then for all $t\in[0;1]$ and for all $k$, the number of semi-infinite bars of $B_k^t$ is constant with respect to the parameter $t$.
\end{cor}

Let us now introduce another space of barcodes which will allow us to get our desired results.

\begin{dfn}\label{def barcodes overall shift}
We define $\hat{\calB}$ as the set of barcodes $\calB$ quotiented by the action by overall shift of $\R$ on $\calB$, i.e. $B$ and $B'$ represent the same class in $\hat{\calB}$ if and only if there is $c\in\R$ such that $B=B'[c]$.
\end{dfn}

Since the action of $\R$ by an overall shift on $\calB$ is free and proper, all the above mentioned topological properties also hold for $\hat{\calB}$.

The only remaining question is the completeness of $\hat{\calB}$. The distance on $\hat{\calB}$ is given by the Hausdorff distance between the equivalence classes which will be denoted $\delta$.

\begin{lemma}
The set $\hat{\calB}$ is complete for the distance $\delta$.
\end{lemma}

\begin{proof}
Let $(\hat{b}_n)_{n\in\N}$ be a Cauchy sequence in $\hat{\calB}$. There is a strictly increasing sequence $(N_p)_{p\in N}$ such that
$$\forall k\in\N,\quad \delta(\hat{b}_{N_p}-\hat{b}_{N_p+k})\leq \frac{1}{2^p}.$$
Let us choose $b_0\in\calB$ a representative of $\hat{b}_{N_0}$ and $b'_1\in\calB$ a representative of $\hat{b}_{N_1}$. Then, by definition of the equivalence classes, there exists $c_1\in\R$ such that $$d_{bottle}(b_0,b'_1[c_1])\leq \frac{1}{2}.$$
Indeed, for all $c\in\R$ and all $b,b'\in\calB$, we have $d_{bottle}(b,b')=d_{bottle}(b[c],b'[c])$.
Now set $b_1=b'_1[c_1]$. We will inductively construct a sequence $(b_p)_{p\in\N}$ such that for all $p$, the barcode $b_p$ is a representative of $\hat{b}_{N_p}$ and $d_{bottle}(b_p,b_{p+1})\leq \tfrac{1}{2^{p+1}}$. Let $p_0\in\N$ and assume that for all $p\in\{0,...,p_0\}$, the barcode $b_p$ is constructed.

The barcode $b_{p_0}$ represents the class of $\hat{b}_{N_{p_0}}$. Let us fix $b'_{p_0+1}$ representing the class of $\hat{b}_{N_{p_0}+1}$. Since $\delta(\hat{b}_{N_{p_0}},\hat{b}_{N_{p_0}+1})\leq \frac{1}{2^{p_0+1}}$, there exists $c_{p_0+1}$ such that $$d_{bottle}(b_{p_0},b'_{p_0+1}[c_{p_0+1}])\leq \frac{1}{2^{p_0+1}}.$$
We define $b_{p_0+1}=b'_{p_0+1}[c_{p_0+1}]$. And thus we obtain our sequence $(b_p)_{p\in\N}$ inductively.

By the triangle inequality \ref{tri ineq} and a classical high school result, we obtain for all $p,k\in\N$
$$d_{bottle}(b_p,b_{p+k})\leq \frac{1}{2^{p}}.$$
Consequently $(b_p)_{p\in\N}$ is a Cauchy sequence which converges to a barcode $b\in\calB$ since $\calB$ is complete. This straightforwardly implies that $(\hat{b}_n)_{n\in\N}$ converges to $\hat{b}$, the equivalence class of $b$, and so $\hat{\calB}$ is complete.
\end{proof}

\subsection{Barcodes for Lagrangian Floer cohomology}\label{sec Floer barcode}

Let $(M,\omega=d\lambda)$ be a Liouville domain, and $L,L'$ two closed exact Lagrangian submanifolds intersecting transversely, together with two primitive functions $f_L:L\rightarrow \R$ and $f_{L'}:L'\rightarrow \R$ such that $df_L=\lambda_{|L}$ and $df_{L'}=\lambda_{|L'}$. We assume that the Floer cohomology is well defined for some Hamiltonian perturbation $H$ and some almost complex structure $J_t$. This is generically satisfied. We assume throughout this subsection that all the Floer cohomologies are well-defined. For all $\kappa\in\R$, we define
$$CF^{*,\kappa}(L,L';J_t,H)=\mathrm{span}_{\Z/2}\left\{z\in\chi(L,L'),~~\calA^H_{L,L'}(z)<\kappa\right\}\subset CF^*(L,L';J_t,H).$$
Let us recall that, for all $x\in CF^{*,\kappa}(L,L';J_t,H)$, we have
$$\calA^H_{L,L'}(\partial x)<\calA^H_{L,L'}(x)<\kappa.$$
This means that $CF^{*,\kappa}(L,L';J_t,H)$ is in fact a subcomplex of $CF^{*}(L,L';J_t,H)$, and consequently we can define:
$$HF^{*,\kappa}(L,L';J_t,H)=H^*(CF^{*,\kappa}(L,L';J_t,H)).$$
Moreover, the inclusions of cochain complexes, i.e. $\forall \kappa'<\kappa\in\R$,
$$CF^{*,\kappa'}(L,L';J_t,H)\subset CF^{*,\kappa}(L,L';J_t,H)$$
induce maps $i^{\kappa'}_\kappa$ in cohomology which commute for $\kappa_1<\kappa_2<\kappa_3$, thus satisfying the property required for structure maps. Finally, $((HF^{*,\kappa}(L,L';J_t,H))_{\kappa\in\R},i)$ has the structure of a finite $\Z$-graded persistence module. We denote its associated graded barcode
$$B(L,L';J_t,H)=B\left((HF^{*,\kappa}(L,L';J_t,H))_{\kappa\in\R},i\right).$$
Since $\chi_H(L,L')$ is finite, $B(L,L';J_t,H)$ is a finite barcode. We will denote $\hat{B}(L,L';J_t,H)$ its image in $\hat{\calB}$.

It is easy to recover the cohomology from the barcode. Indeed, by definition
$$\lim\limits_{\underset{\kappa\to\infty}{\rightarrow}}CF^{*,\kappa}(L,L';J_t,H)=CF^{*}(L,L';J_t,H)$$ and then
$$\lim\limits_{\underset{\kappa\to\infty}{\rightarrow}}HF^{*,\kappa}(L,L';J_t,H)=HF^{*}(L,L';J_t,H).$$
This means that $HF^{*}(L,L';J_t,H)$ corresponds to the bars that survive when $\kappa$ goes to infinity, i.e.
\begin{rk}\label{prop homologie barres infinies}
The graded rank of $HF(L,L';J_t,H)$ is equal to the graded number of semi-infinite bars in $B(L,L';J_t,H)$.
\end{rk}
We can now recall the definition of selectors. This selector, denoted by $l(\cdot,L,L';J_t,H)$, can be understood as the action selector of the persistence module $HF^\kappa(L,L';J_t,H)$. Let us give an explicit definition.
\begin{dfn}\label{def selc action lag}
To any $\alpha\in HF^*(L,L';J_t,H)\setminus\{0\}$, we associate
\begin{equation}\label{eqn action selectors}
l(\alpha,L,L';J_t,H)=\inf\{\kappa\in\R,~~ \alpha\in \im i^\kappa: HF^{*,\kappa}(L,L';J_t,H)\rightarrow HF^{*}(L,L';J_t,H)\}.
\end{equation}
\end{dfn}

These numbers are exactly all the different starting points of the semi-infinite bars, i.e. each semi-infinite bar corresponds to some non-zero $\alpha\in HF^*(L,L';J_t,H)$, and the starting point of this particular semi-infinite bar is given by $l(\alpha, L,L';J_t,H)$.

\bigskip

The following proposition gives classical properties of these action selectors as found in \cite{Vit92,Sch00,Oh04,Lec07}.
\begin{prop}\label{prop propriétés action selectors}
For every pair of closed exact Lagrangian submanifolds in a Liouville domain, and every non-zero class $\alpha\in HF(L,L';J_t,H)$, the action selector $l(\alpha,L,L';J_t,H)$ satisfies:
\begin{itemize}
    \item $l(\alpha,L,L';J_t,H)<+\infty$,
    \item $l(\alpha,L,L';J_t,H)\in Spec(L,L';H)$,
    \item $l(\alpha,L,L';J_t,H)$ does not depend on $J_t$ hence will be denoted $l(\alpha,L,L';H)$,
    \item $|l(\alpha,L,L';H)-l(\alpha,L,L';H')|\leq \|H-H'\|$, where $\|\cdot\|$ denotes the Hofer norm.
\end{itemize}
\end{prop}
The second property is called the spectrality property, and the fourth one the Lipschitz continuity property. These are classical results when studying action selectors and thus we will not prove them here. However, we can say that the first three properties directly follow from the definition. The fourth one is a direct consequence of the construction of continuation maps used to prove that the cohomology does not depend on the choice of the Hamiltonian perturbation.

These action selectors satisfy the so-called Lagrangian splitting formula which is a direct consequence of the Künneth formula (\ref{formule def Kunneth}); see for example \cite{EP09} or \cite{HLS16}.

\begin{prop}\label{prop formule de Kunneth}
Let $(M,\omega)$ and $(M',\omega')$ be symplectic manifolds as before, and $(L_0,L_1)\subset M$, $(L_0',L_1')\subset M'$ two pairs of closed exact Lagrangian submanifolds. Let $H$ and $H'$ be two Hamiltonian perturbations to achieve transversality. Then, for $\alpha\in HF(L_0,L_1;J_t,H)$ and $\alpha'\in HF(L_0',L_1';J'_t,H')$ two non-zero cohomology classes,
$$l(\alpha\otimes\alpha';L_0\times L_0',L_1\times L_1';H\oplus H')=l(\alpha,L_0,L_1;H)+l(\alpha',L_0',L_1';H'),$$
where $\alpha\otimes\alpha'$ is defined by the Künneth formula (\ref{formule def Kunneth}).
\end{prop}

The continuation maps in Floer cochain complexes give the continuity of the barcodes with respect to the Hofer distance:
\begin{prop}\label{prop barcode distance Hofer}
Let $L,L'$ be two closed exact Lagrangian submanifolds in a Liouville domain, and let $H,K$ be two Hamiltonians together with time dependent almost-complex structure $J$ and $J'$ such that the graded barcodes $B^*(L,L';J_t,H)$ and $B^*(L,L';J'_t,K)$ are well-defined. Then,
$$d_{bottle}(B(L,L';J_t,H),B(L,L';J'_t,K))\leq \|H-K\|,$$
where $\|.\|$ denotes the Hofer distance.
\end{prop}
Note that this bound does not depend on choice of the almost complex structures $J$ and $J'$.

The proof of this proposition is a straightforward translation to our context of a well-known result proven by Polterovich-Shelukhin \cite{PolShel14} and Usher-Zhang \cite{UZ15} in full generality.

In the following sections, we do not really care about the Hamiltonian perturbation. The fact that given any two closed exact Lagrangian submanifolds, the Hamiltonian perturbation can be made as small as one wishes by \autoref{rk perturbation petite}, together with the \autoref{prop propriétés action selectors} allows us to define, for $a$ a non-zero class in $HF(L,L';J_t,H)$
$$l(a;L,L')=\lim\limits_{H\in\mathcal{H}\to 0}l(a;L,L';H),$$
where $\mathcal{H}$ is the set of Hamiltonians satisfying the transversality requirements. If $L$ and $L'$ intersect transversely, it equals the action selector defined in \autoref{def selc action lag}.

We can also use for barcodes the perturbation data notation as in Section \ref{sec preli}, i.e. denoting $\calD$ the pair $(H,J)$ where $H$ is the Hamiltonian perturbation and $J$ the regular almost complex structure, the barcode can be written
$$B(L,L';\calD).$$
We will denote $\hat{B}(L,L';\calD)$ its image in $\hat{\calB}$.

Following \autoref{prop barcode distance Hofer}, given two closed exact Lagrangian submanifolds $L,L'$ in a Liouville domain $(M,\omega=d\lambda)$ with two primitive functions $f_L:L\rightarrow \R$ and $f_{L'}:L'\rightarrow \R$ such that $df_L=\lambda_{|L}$ and $df_{L'}=\lambda_{|L'}$, the map
$$H\mapsto\calB(L,L';H,J)$$
is continuous with respect to the Hofer distance. Since the space of barcodes is complete by \autoref{prop barcode complet},
we can take the limit of $B(L,L';\calD)$ as the Hamiltonian part of the perturbation goes to zero and thus define
$$B(L,L')=\lim\limits_{H\to 0}B(L,L';H,J).$$

For two exact Lagrangian submanifolds $L$ and $L'$, we denote $\hat{B}(L,L')$ the image of $\hat{B}(L,L')$ in $\hat{\calB}$.

\subsection{Product in filtered Lagrangian Floer cohomology}\label{sec product barcode}

In this section, we focus on the action for a product on Floer complexes. Regarding the degree, results are the same as those in non-filtered Floer cohomology. However we need to understand precisely how we can bound the shift of action in order to define this structure on filtered Floer cohomology.

\bigskip

Let $(M,\omega)$ be a $2n$-dimensional exact symplectic manifold. Let $L_0$, $L_1$, $L_2$ be three pairwise transverse closed exact Lagrangian submanifolds in $M$. We assume that the product is well defined.

Since these Lagrangian submanifolds are exact, they come with three primitive functions (defined up to a constant) $f_i:L_i\mapsto \R$, such that $df_i=\lambda_{|L_i}$ for $i\in\{0,1,2\}$.
Let $p_1\in \chi(L_0,L_1)$, $p_2\in \chi(L_1,L_2)$ and $z\in CF^*(L_0,L_2)$ such that $\mu^2(p_2,p_1)=z$. Note that $z$ is a formal sum of $(q_j)_j\in\chi(L_0,L_2)$.

\begin{figure}[ht]
    \centering
    \includegraphics[scale=1.2]{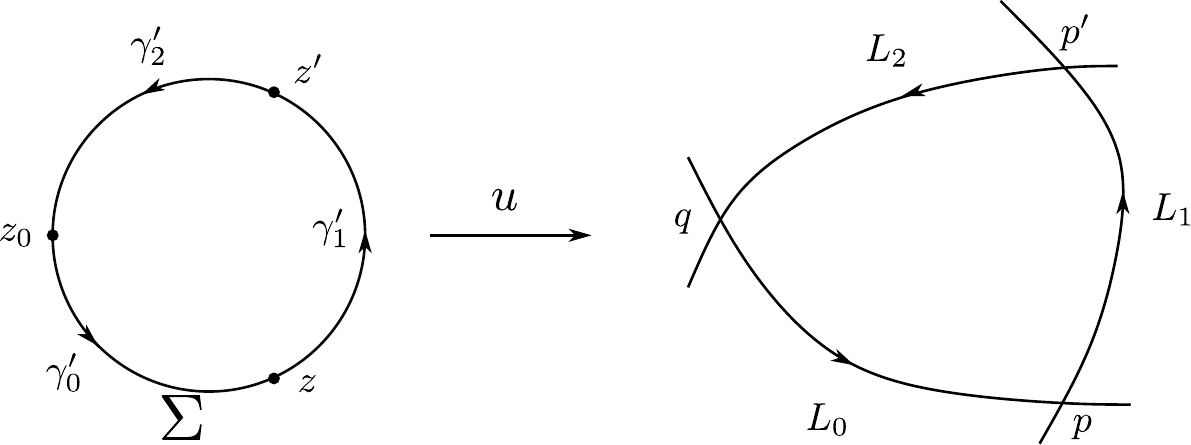}
    \caption{Product in Floer cohomology}
    \label{fig product filtered Floer}
\end{figure}

Let us recall that
$$\calA_{L_0,L_1}(p_1)=f_0(p_1)-f_1(p_1),$$
$$\calA_{L_1,L_2}(p_2)=f_1(p_2)-f_2(p_2),$$
$$\calA_{L_0,L_2}(q_i)=f_0(q_i)-f_2(q_i).$$

Let $u:\Sigma\rightarrow (M;L_0,L_1,L_2)$ be a pseudo-holomorphic curve with punctures asymptotic to $(p_1,p_2,q_j)$ as classically defined for the product in Lagrangian Floer cohomology. Let us denote for $i\in\{0;1;2\}$ the paths $\gamma_i :[0;1]\rightarrow L_i$ such that $\gamma_i([0;1])= u(\D^2,\partial\D^2)\cap L_i$. We set the orientations of the $\gamma_i$ for $i\in\{0;1;2\}$ such that their concatenation $\gamma_0\sharp\gamma_1\sharp\gamma_2$ turns counterclockwise as in \autoref{fig product filtered Floer}.

Since $\omega$ is exact, equal to $d\lambda$, Stokes' theorem gives
\begin{eqnarray*}
\mathrm{Area}(u)&=& \int_{\D^2}u^*\omega\\
&=&\int_{\gamma_1}\lambda_{L_1}+\int_{\gamma_2}\lambda_{L_2}+\int_{\gamma_0}\lambda_{L_0}.
\end{eqnarray*}
Moreover, all the $L_i$ being exact Lagrangian submanifolds, with associated functions $f_i$, we get:

$$\forall i\in \{0,1,2\}, \int_{\gamma_i}\lambda_i=f_i(\gamma_i(1))-f_i(\gamma_i(0)).$$
Then,
\begin{eqnarray*}
\mathrm{Area}(u)&=&f_0(p_1)-f_0(q_j)+f_1(p_2)-f_1(p_1)+f_2(q_j)-f_2(p_2)\\
&=&\calA_{L_0,L_1}(p_1)+\calA_{L_1,L_2}(p_2)-\calA_{L_0,L_2}(q_j).
\end{eqnarray*}

Since the area of $u$ is positive, we have
$$\calA_{L_2,L_0}(q_j)<\calA_{L_0,L_1}(p_1)+\calA_{L_1,L_2}(p_2).$$
Let us recall that
$$\calA_{L_2,L_0}(z)=\max\limits_{j}\{\calA_{L_2,L_0}(q_j)\}.$$
We immediately get
$$\calA_{L_2,L_0}(z)<\calA_{L_0,L_1}(p_1)+\calA_{L_1,L_2}(p_2).$$

As done in previous sections, we now have to discuss the case where we do not assume the transversality properties, and hence where we need a perturbation data $\calD$. The argument is exactly the same as for Inequality (\ref{energie action avec perturbation}), as the perturbation data has to be taken into account in the same way when computing $E(u)$. Since the perturbation data $\calD$ can be chosen as small as desired, as before, for all $\varepsilon>0$, we can find $\calD$ such that all our cohomologies are well-defined and
$$E(u)\leq\calA_{L_0,L_1}(p_1)+\calA_{L_1,L_2}(p_2)-\calA_{L_0,L_2}(z)+\varepsilon.$$
We then straightforwardly obtain
$$\calA_{L_2,L_0}(z)\leq\calA_{L_2,L_1}(p_1)+\calA_{L_1,L_0}(p_2)+\varepsilon,$$
for $p_1\in \chi_{\mathcal{D}}(L_0,L_1)$, $p_2\in \chi_{\mathcal{D}}(L_1,L_2)$ and $z\in CF^*(L_0,L_2;\calD)$ with $\mu^2(p_2,p_1)=z$.

This means that the product preserves the filtration and immediately implies the following lemma which will be essential for the upcoming discussions.

\begin{lemma}\label{lem mult pers mod}
Let $L_0$, $L_1$, $L_2$ be three closed exact Lagrangian submanifolds in $(M,\omega)$ exact, together with a $\varepsilon$-small perturbation data collection $\calD$, and let $p_2\in CF^k(L_1,L_2;\calD)$, with action $b$. Let us assume that the product $\mu^2(p_2,\cdot):CF^*(L_0,L_1;\calD)\mapsto CF^{*+k}(L_0,L_2;\calD)$ is well defined.

Then, we have a morphism of persistence modules:
$$\mu^2(p_2,\cdot):CF^{*,t}(L_0,L_1;\calD)\rightarrow CF^{*+k,t+b+\varepsilon}(L_0,L_2;\calD),\quad \forall t\in\R$$
$$\mu^2(p_2,\cdot):CF^{*}(L_0,L_1;\calD)\rightarrow CF^{*+k}(L_0,L_2;\calD)[b+\varepsilon].$$
\end{lemma}

\begin{lemma}\label{lem associativite barcodes}
Let $L_0$, $L_1$, $L_2$ be three closed exact Lagrangian submanifolds in $(M,\omega)$ exact together with a perturbation data $\varepsilon$-small $\calD$. Let $p_1\in CF^k(L_0,L_1;\calD)$, with action $a$, $p_2\in CF^k(L_1,L_0;\calD)$, with action $b$.
The following maps obtained by composition
$$\mu^2(p_2,\mu^2(p_1,\cdot)):CF(L_2,L_0;\calD)\rightarrow CF(L_2,L_0;\calD)[a+b+3\varepsilon],$$
$$\mu^2(p_1,\mu^2(p_2,\cdot)):CF(L_2,L_1;\calD)\rightarrow CF(L_2,L_1;\calD)[a+b+3\varepsilon],$$
are well-defined and filtered chain homotopic to the maps
$$\mu^2(\mu^2(p_2,p_1),\cdot):CF(L_2,L_0;\calD)\rightarrow CF(L_2,L_0;\calD)[a+b+3\varepsilon],$$
$$\mu^2(\mu^2(p_1,p_2),\cdot):CF(L_2,L_1;\calD)\rightarrow CF(L_2,L_1;\calD)[a+b+3\varepsilon].$$
\end{lemma}

\begin{proof}
The composition maps are well-defined and filtered by the preceding lemma. Since the product in Lagrangian Floer cohomology is associative (following from the next equality), we only have to check that the associator behaves correctly with respect to the filtration. Let us recall that for our chain complexes we have
\begin{eqnarray*}
\mu^2(\mu^2(p_2,p_1),q)+\mu^2(p_2,\mu^2(p_1,q))&  = \partial \mu^3(p_2,p_1,q)
 +\mu^3(\partial p_2,p_1,q)\\ &+\mu^3(p_2,\partial p_1,q)+\mu^3(p_2,p_1,\partial q),
\end{eqnarray*}
where $q$ is an element of $CF(L_2,L_0)$.
Then, the exact same computation as for $\mu^2$ gives us
$$\calA_{L_2,L_0}(\mu^3(p_2,p_1,q))\leq \calA_{L_0,L_1}(p_1)+\calA_{L_1,L_0}(p_2)+\calA_{L_2,L_0}(q)+3\varepsilon.$$
Moreover, the differential decreases the action, so that
\begin{eqnarray*}
\max\{\calA_{L_2,L_0}(\partial \mu^3(p_2,p_1,q)),\calA_{L_2,L_0}(\mu^3(\partial p_2,p_1,q)),\\ \calA_{L_2,L_0}(\mu^3(p_2,\partial p_1,q)),\calA_{L_2,L_0}(\mu^3(p_2,p_1,\partial q))\}\\
\leq \calA_{L_0,L_1}(p_1)+\calA_{L_1,L_0}(p_2)+\calA_{L_2,L_0}(q)+3\varepsilon.
\end{eqnarray*}
This means that the homotopy defined from $\mu^3$ between the two different compositions preserves the filtration, which concludes the proof of this lemma.
\end{proof}

The following lemma will be a key argument in the proof of Section \ref{sec bottleneck}.
\begin{lemma}\label{lem mult id barcodes}
Let $L,L'$ be two closed exact Lagrangian submanifolds in $(M,\omega)$ together with $\varepsilon$-small perturbation data $f$ and let $H$ behave as in \autoref{prop identite mult z}. Denote $H_f=f\sharp H$. Let $z\in CF^0(L',L';f,J)$ be as in the same proposition.
The multiplication map
$$m_2(\cdot,z):CF^*(L',L;H)\rightarrow CF^*(L',L;H_f)[2\varepsilon]$$
are filtered chain-homotopic to the standard inclusion and hence induce $2\varepsilon$-shift maps on the persistence modules.
\end{lemma}

\begin{proof}
Let us recall that \autoref{prop identite mult z} tells us the multiplication by $z$ is an isomorphism of cochain complexes and hence induces the standard inclusion of persistence modules. We now just need the energy estimate.

Since the Hamiltonian part of the perturbations are $\varepsilon$-small, the action of $z$ is smaller than $\varepsilon$ and $H_f$ is $\varepsilon$ close to $H$. Consequently, using the same argument as the one implying \autoref{lem mult pers mod}, this map induces a $\varepsilon+\varepsilon=2\varepsilon$-shift of action. This concludes the proof of this lemma.
\end{proof}

\subsection{Spectral norm and exact Lagrangians in a cotangent bundle}\label{sec gamma}

Given a closed exact Lagrangian submanifold $L$ together with a non-degenerate Hamiltonian $H$, the spectral norm $\gamma_L(H)$ is defined as 
$$\gamma_L(H)=l([L],L,L;H)+l([L],L,L;\overline{H}),$$
where $[L]$ denotes the image of the fundamental class $[L]$ through the isomorphism of \autoref{prop lien entre HF et H}.
It is equal to the diameter of the spectrum $\Spec(L,L;H)$. This is called the Lagrangian spectral norm or Viterbo norm as its first version was introduced by Viterbo in \cite{Vit92}. A similar version also exists in Hamiltonian Floer homology.

Let $L$ and $L'$ be two closed exact Lagrangian submanifolds in a symplectic manifold $M$ as before, together with a Hamiltonian perturbation $H$. 
Then, in the same spirit, we set
$$\gamma(L,L';H)=\mathrm{Diam}(\Spec^*(L,L';H)),$$
where $\Spec^*(L,L';H)$ is the set of action selectors for $HF(L,L';H,J)$. We denote by $\mathrm{Diam}(\cdot)$ the diameter (i.e. $\max-\min$). Note that this definition is only interesting when the cohomology $HF(L,L';H,J)$ has rank at least 2.
By \autoref{prop formule de Kunneth}, and with the same notations we immediately get
\begin{equation}\label{eq gamma kunneth}
\gamma(L_0\times L_0',L_1\times L_1';H\oplus H')=\gamma(L_0,L_1;H)+\gamma(L_0',L_1';H').
\end{equation}
Consequently, if $M=M'$, $L_0=L_0'$, $L_1=L_1'$ and $H=H'$,
\begin{equation}
\gamma(L_0\times L_0,L_1\times L_1;H\oplus H)=2\gamma(L_0,L_1;H).    
\end{equation}

\begin{rk}\label{rk majoration gamma fonction}
Let $L$ and $L'$ be two closed exact Lagrangian submanifolds in a Liouville domain $(M,\omega)$ together with a Hamiltonian $H$ and a function $f:M\rightarrow \R$. The fourth point of \autoref{prop propriétés action selectors} together with the definition of $\gamma$ tells us that
$$|\gamma(L,L';H+f)-\gamma(L,L';H)|\leq 2(\max f -  \min f).$$
\end{rk}

Note that we do not need any transversality assumptions for the intersections between $L$ and $L'$. Indeed, by the continuity of the spectral invariants, $\gamma(L,L';H)$ is defined for all $H$. So we can set
$$\gamma(L,L')=\lim\limits_{H\in\mathcal{H}\to 0}\mathrm{Diam}(\Spec^*(L,L';H)),$$
where $\mathcal{H}$ is the set of Hamiltonians satisfying the transversality requirements.

\begin{rk}
Given two Lagrangian submanifolds $L$ and $L'$, we can actually define $\gamma$ directly since the spectrum is defined without any transversality assumptions.
\end{rk}

The question of the continuity of $\gamma$ with respect to the $C^0$-distance is a fundamental one. This has been proved for specific symplectic manifolds in \cite{Vit92,Sey13,BHS19,Kaw19,Shel19}. We will also prove its continuity in our context in Section \ref{sec gamma C0}, and thus we will not discuss it more here.

\bigskip

An important question is the relation between two $C^0$-close Lagrangian submanifolds. It is related to Arnold's famous Nearby Lagrangian Conjecture. The results on this conjecture will be useful for both another definition of $\gamma$ and for the arguments of Section \ref{sec bottleneck}. Let us start by stating this conjecture.
\begin{conj}
Let $M$ be a closed connected manifold. Any closed exact connected Lagrangian submanifold in $T^*M$ is Hamiltonian isotopic to the zero section.
\end{conj}
Together with Weinstein's theorem, this conjecture implies that in a symplectic manifold $M$ together with a closed connected Lagrangian submanifold $L\subset M$, any exact closed connected Lagrangian submanifold $L'\subset M$ is Hamiltonian isotopic to $L$ if $L'$ is $C^0$-close enough to $L$.

For most cases, this conjecture is still open and subject to a lot of research. It has been fully proved in special cases. Hind \cite{H04} proved the following theorem:
\begin{thm}\label{thm NLC S2}
The Nearby Lagrangian Conjecture is true in $T^*S^2$.
\end{thm}
For $T^*S^1$, there is not much to discuss and it is also true. Dimitroglou-Rizell, Goodman and Ivrii \cite{RGI16} proved it for $T^*\T^2$.

In more general context, important progress has been made by Fukaya, Seidel and Smith \cite{FSI07}, proving that when the Maslov class vanishes, the projection
$$\pi:L'\rightarrow L\subset T^*L$$
induces an isomorphism on homology. This result was improved later by Abouzaid-Kragh in the following theorem \cite{AK18}.
\begin{thm}\label{thm NLC}
Let $L$ be a closed connected manifold together with $L'$ an exact closed connected Lagrangian submanifold in $T^*L$. Then, there exists an integer $i=i_{L'}\in\Z$ such that for every exact closed Lagrangian submanifold $K$ in $T^*L$, there are chain-level quasi-isomorphisms in both directions between $CF^*(L',K)$ and $CF^{*+i}(L,K)$ and between $CF^*(K,L')$ and $CF^{*-i}(K,L)$. These quasi-isomorphisms are compatible with the product structure in Floer cohomology.
\end{thm}

 Let $L_0$ and $L_1$ be two closed exact Lagrangian submanifolds in $(T^*L,\omega=d\lambda)$ exact, together with two primitives functions $f_{L_0}$ and $f_{L_1}$ such that $df_{L_i}=\lambda_{|L_i}$, for $i\in\{0,1\}$. As mentioned before, this theorem allows us to state another definition of $\gamma(L_0,L_1)$. Indeed, previous \autoref{thm NLC} and \autoref{prop lien entre HF et H} respectively tell us that we have the two following isomorphisms
$$HF^*(L,L)\xrightarrow[\sigma]{\sim} HF^*(L_0,L_1),$$
$$H_*(L) \xrightarrow[\theta]{\sim} HF^{n-*}(L,L).$$
By abuse of notation, we denote
$$[L]=\sigma\circ\theta([L])\in HF^0(L_0,L_1),$$
$$[pt]=\sigma\circ\theta([pt])\in HF^n(L_0,L_1).$$

It is known that for such Lagrangian submanifolds $L_0$ and $L_1$, the quantity $\gamma(L_0,L_1)$ admits the following alternative definition. This definition is actually the standard definition of $\gamma(L_0,L_1)$; see \cite{MVZ12,LZ18}.
\begin{dfn}\label{lem def gamma}
$$\gamma(L_0,L_1)=l\left([L],L_0,L_1\right)-l\left([pt],L_0,L_1\right).$$
\end{dfn}

Let us give the basic properties of $\gamma$. Following \autoref{def action}, we have $\calA_{L_0,L_1}=-\calA_{L_1,L_0}$. Together with the fact that the two complexes $CF(L_0,L_1)$ and $CF(L_1,L_0)$ are dual to each other, we have
$$l\left([pt],L_0,L_1\right)=-l\left(\sigma'\circ\theta([L]),L_1,L_0\right),$$
where $\sigma'$ is the isomorphism from $HF(L,L)$ to $HF(L_1,L_0)$ given by \autoref{thm NLC}. We thus obtain
$$\gamma(L_0,L_1)=l\left([L],L_0,L_1\right)+l\left(\sigma'\circ\theta([L]),L_1,L_0\right).$$
Consequently, for all $L_0$ and $L_1$ exact in $T^*L$, 
\begin{equation}\label{eq sym gamma}
\gamma(L_0,L_1)=\gamma(L_1,L_0).    
\end{equation}
Moreover, for all $L_0,L_1,L_2$ closed exact Lagrangian submanifolds in $T^*L$ with primitive functions $f_{L_0},f_{L_1},f_{L_2}$, it also satisfies the triangle inequality
\begin{equation}\label{eq gamma triang}
    \gamma(L_0,L_1)\leq\gamma(L_0,L_2)+\gamma(L_2,L_1).
\end{equation}
Indeed, if $x\in CF(L_2,L_1)$ and $y\in CF(L_0,L_2)$ both represent the fundamental class in their respective homology, so does $\mu^2(x,y)$ in $CF(L_0,L_1)$ (see Section \ref{sec bottleneck}). Together with \autoref{lem mult pers mod}, we immediately obtain this triangle inequality.
\section{Continuity of the barcode}\label{sec barcode continu}

\subsection{Results and idea of the proof}\label{sec result idea proof}
\subsubsection{Main theorem and consequences}\label{subsec intro proof}
The object of this section is to prove the following theorem, corresponding to \autoref{thmx loc lip} in the Introduction which will be the key to prove our results concerning the Dehn-Seidel twist. It shows a certain local Lipschitz continuity on barcodes associated to Lagrangian submanifolds. We will always assume that the considered Lagrangian submanifolds are connected.

\begin{thm}\label{prop contL barcode}
Let $M$ be a Liouville domain. Let $L$ and $L'$ be two closed exact Lagrangian submanifolds, and assume that $H^1(L',\R)=0$. Then there exist $K\geq0$ and $l>0$ such that for all $\varphi$ and $\psi$ in $\Symp(M,\omega)$, if $d_{C^0}(\varphi, \psi) \leq l$, we have $$d_{bottle}(\hat{B}(\varphi(L'),L),\hat{B}(\psi(L'),L))\leq K d_{C^0}(\varphi, \psi).$$
\end{thm}
The fact that we have a uniform Lipschitz continuity with respect to the $C^0$ distance immediately implies the following corollary.
\begin{cor}\label{Cor extension}
The map $\varphi \mapsto \hat{B}(\varphi(L'),L)$ continuously extends to a map $\overline{\mathrm{Symp}}(M,\omega)\rightarrow\hat{\calB}$.
\end{cor}

Since $L$ and $L'$ are closed, the number of semi-infinite bars of $B(\varphi(L'),L)$ stays finite for all $\varphi\in\Symp(M,\omega)$. This extension to the closure requires to work with $\calB$ as defined in \autoref{def Q-tame barcodes} which is the completion of the space of barcodes $\calB_f$ by \autoref{prop barcode complet}. As we will see in the proof, we will then have to work in the space of barcodes up to shift $\hat{\calB}$.

From \autoref{prop contL barcode}, we obtain the following theorem. It is direct consequence of the continuity of barcodes together with its \autoref{Cor extension}.
\begin{thm}\label{thm barcodes connected}
Let $M$ be a Liouville domain. Let $L$ and $L'$ be two exact compact Lagrangian submanifolds, and assume that $H^1(L',\R)=0$. Consider two symplectomorphisms $\varphi$ and $\psi$.
\begin{itemize}
    \item If these two symplectomorphisms are in the same connected component of $\overline{\Symp}(M,\omega)$, then the two barcodes $\hat{B}(\varphi(L'),L)$ and $\hat{B}(\psi(L'),L)$ are in the same connected component of $\hat{\calB}$.
    \item If these two symplectomorphisms are isotopic in $\overline{\Symp}(M,\omega)$, then there is a continuous path of barcodes from $\hat{B}(\varphi(L'),L)$ to $\hat{B}(\psi(L'),L)$.
\end{itemize}
\end{thm}

This continuous path can also be directly constructed in the following way from \autoref{Cor extension}. Let us denote $(\phi^t)_{t\in[0,1]}$ the path in $\overline{\Symp}(M,\omega)$ from $\varphi$ to $\psi$. For each $t\in[0,1]$, \autoref{Cor extension} allows to associate a barcode $\hat{B}^t$ to $\phi^t$. The path of barcodes is then the path $(\hat{B}^t)_{t\in[0,1]}$.

The second point of \autoref{thm barcodes connected} can also be understood as a consequence of the first point as $\hat{\calB}$ is locally path-connected.

\begin{rk}
In smooth symplectic topology, the two points in \autoref{thm barcodes connected} would be equivalent. However, in $C^0$ symplectic topology we do not know whether $\overline{\Symp}(M,\omega)$ is locally path-connected, thus it is not known whether the connected components of $\overline{\Symp}(M,\omega)$ are path-connected. Consequently the first point implies the second one but the reciprocal implication is far from clear.
\end{rk}

\subsubsection{Structure of the proof of \autoref{prop contL barcode}}\label{subsec idea proof}

In order to prove \autoref{prop contL barcode}, we prove the two following propositions. The first one bounds the bottleneck distance by the spectral norm $\gamma$.
\begin{prop}\label{cor maj bottleneck}
Let $L$ and $L'$ be two closed exact Lagrangian submanifolds in a Liouville domain $(M,\omega)$. There exists $\delta>0$, independant of $L$, such that for all  $\varphi\text{ and }\psi\in\Symp(M,\omega)$ satisfying $d_{C^0}(\varphi,\psi)\leq\delta$, then
$$d_{bottle}(\hat{\calB}(\varphi(L'),L),\hat{\calB}(\psi(L'),L))\leq  \tfrac{1}{2}\gamma(L',\psi^{-1}\circ\varphi(L')).$$
\end{prop}
This proposition is an adaptation to our context of a similar statement proved by Kislev and Shelukhin \cite{KS18}. In their case, $L=L'$ is a weakly monotone Lagrangian submanifold in a closed symplectic manifold and $\varphi$ and $\psi$ are Hamiltonian diffeomorphisms.

This proposition will be proven in Subsection \ref{sec bottleneck}.
The second proposition asserts that $\gamma(L',\varphi(L'))$ goes to zero, as $\varphi$ goes to identity and will be proven in Subsection \ref{sec gamma C0}.
\begin{prop}\label{lem gamma C0}
There exist constants $l\geq 0$ and $\kappa\geq0$ such that and for all $\varphi\in\Symp(M,\omega)$ satisfying $d_{C^0}(\varphi,\Id_M)\leq l$, we have
$$\gamma(L',\varphi(L'))\leq \kappa d_{C^0}(\varphi,\Id_M).$$
\end{prop}
This proposition is an adaptation to our context of a lemma of Buhovsky-Humilière-Seyfaddini \cite{BHS19}. In their paper, they proved the same result for a Lagrangian submanifold Hamiltonian isotopic to the zero section in a cotangent bundle.

\begin{proof}[Proof of \autoref{prop contL barcode}]

Let  $\varphi$ and $\psi$ be in $\Symp(M,\omega)$ such that $d_{C^0}(\varphi, \psi) \leq l$. We can assume without loss of generality that $l\leq\delta$. (See the choice of $l$ in Section \ref{sec gamma C0}.)

Indeed we have
\begin{eqnarray*}
d_{bottle}(\hat{\calB}(\varphi(L'),L),\hat{\calB}(\psi(L'),L)) & \leq & \tfrac{1}{2}\gamma(L',\psi^{-1}\circ\varphi(L'))\\
&\leq & \tfrac{1}{2} \kappa d_{C^0}(\psi^{-1}\circ\varphi,\Id_M)\\
&=&  \tfrac{1}{2} \kappa \sup\limits_{x\in M} d(\psi^{-1}(x),\varphi^{-1}(x))\\
&\leq&  \tfrac{1}{2} \kappa d_{C^0}(\psi,\varphi).
\end{eqnarray*}
Setting $K=\tfrac{1}{2}\kappa$, this proves \autoref{prop contL barcode}.
 
\end{proof}

\bigskip

Let us now briefly sketch the proof of \autoref{cor maj bottleneck} and \autoref{lem gamma C0} and set up some conventions. \autoref{cor maj bottleneck} will be implied by the case where $\psi=\Id_M$.

Let us fix $\varepsilon_0>0$, $\varepsilon'\ll\varepsilon_0$ and assume that all the Hamiltonian parts of the perturbation data at stake in this proof are of $C^2$-norm smaller than $\varepsilon'$.

Let us fix such a perturbation data collection $\calD$ such that $HF^t(\varphi(L'),L; \calD)$,  $HF^t(L',L;\calD)$, $HF^t(\varphi(L'),L'; \calD)$, $HF^t(L',L';\calD)$ and $HF^t(\varphi(L'),\varphi(L');\calD)$ are well defined.

\begin{rk}\label{rk choix perturbation data}
 In the case of $HF(L',L';\calD)$, we require that the Hamiltonian perturbation is defined in the following way (see also \autoref{rk representant classe fondamentale unique}). Let $f$ be a $\varepsilon'/2$-small Morse function defined on $L'$ with a unique maximum and a unique minimum. We extend it to a Hamiltonian $H$ which is supported on a $\varepsilon_0$-small tubular neighbourhood of $L'$. This construction implies that there is only one element in $CF^n(L',L';\calD)$ and only one in $CF^0(L',L';\calD)$. We perform the same construction in the case of $HF(\varphi(L'),\varphi(L');\calD)$. 
\end{rk}

We aim to find two morphisms of persistence modules $A=\{A^t\}_{t\in\R}$ and $B=\{B^t\}_{t\in\R}$ together with $\delta,\delta'\in\R$:
$$A^t:CF^t(\varphi(L'),L;\calD)\longmapsto CF^{t+\delta}(L',L;\calD),$$
$$B^t:CF^t(L',L;\calD)\longmapsto CF^{t+\delta'}(\varphi(L'),L;\calD),$$
such that these maps are filtered and their compositions are chain homotopic to shifts of persistence modules:
$$sh_{\varphi(L')}:\calB(\varphi(L'),L;\calD)\longmapsto\calB(\varphi(L'),L;\calD)[\delta+\delta'+\varepsilon']$$
$$sh_{L'}:\calB(L',L)\longmapsto\calB(L',L)[\delta+\delta'+\varepsilon'].$$
If they indeed satisfy the above conditions, these maps $A$ and $B$ provide a $\delta+\delta'+\varepsilon'$-matching. Then, to achieve the proof, we will only have to bound the shift $\delta+\delta'+\varepsilon'$ by the $C^0$ distance between $\varphi$ and $\Id_M$. We will prove that this shift is in fact equal to $\frac{1}{2}\gamma(L';\varphi(L');\calD)+\varepsilon'$, and use this to get the bound. This is the purpose of Section \ref{sec bottleneck}. Proving that this bound goes to zero when $\varphi$ $C^0$-converges to the identity is the purpose of the last Section \ref{sec gamma C0}.

Following Kislev and Shelukhin's idea \cite{KS18}, these maps $A$ and $B$ will come from the multiplication in Floer cohomology:
\begin{itemize}
\item $A$ corresponds to the multiplication by a specific class $[x]$ in $HF(L',\varphi(L');\calD)$.
\item $B$ corresponds to the multiplication by a specific class $[y]$ in $HF(\varphi(L'),L';\calD)$.
\end{itemize}
These choices will be achieved using Abouzaid-Kragh's \autoref{thm NLC} \cite{AK18}. This result requires the Lagrangian submanifolds to be in a cotangent space. To obtain this requirement, we will consider a symplectomorphism $\varphi$ $C^0$-close enough to the identity so that $\varphi(L')$ is included in a Weinstein neighbourhood of $L'$. We thus obtain two cohomologies which could be different: the one computed in $M$ and the one computed in $T^*L'$. Consequently, for the sake of our argument, we will first prove that we have the isomorphisms $$HF(L',\varphi(L');\calD,M)\cong HF(L',\varphi(L');\calD,T^*L'),$$ $$HF(\varphi(L'),L';\calD,M)\cong HF(\varphi(L'),L';\calD,T^*L').$$ Of course we will also prove that these isomorphisms respect the filtration. We will in fact only give the details for one of these isomorphisms since the proofs are identical for both. By abuse of notation, we denote by $\calD$ both the perturbation data in $M$ and its image in $T^*L$. This is the purpose of the following Section \ref{sec barcode cotangent}.

\begin{rk}\label{rk exacteness}
Now that the proof is sketched, we can explain the conditions required for the two Lagrangian submanifolds $L'$ and $L$ in \autoref{prop contL barcode}. These are both exactness conditions. In the previous chapters, to define Lagrangian Floer cohomology, the product and the action filtration, we require the considered Lagrangian submanifolds to be exact. This exactness condition is also required for \autoref{thm NLC} that will be used to construct the maps $A$ and $B$.

The condition $H^1(L',\R)=0$ guarantees that, for any symplectomorphism $\varphi\in \Symp(M,\omega)$, $\varphi(L')$ is an exact Lagrangian submanifold as well. With these conditions, we are sure that all the above mentioned objects used in the following proof will be well defined.

This implies that, when working with $\varphi\in\Ham(M,\omega)$, we can drop the condition $H^1(L',\R)=0$ for the weaker condition that $L'$ is exact. Indeed, the image of an exact Lagrangian submanifold by a Hamiltonian diffeomorphism is always exact.
\end{rk}

The following Subsections \ref{sec barcode cotangent} and \ref{sec bottleneck} are dedicated to the proof of \autoref{cor maj bottleneck} and Subsection \ref{sec gamma C0} to the proof of \autoref{lem gamma C0}.

\subsection{Equality of the barcodes in $M$ and in $T^*L'$}\label{sec barcode cotangent}

If the symplectomorphism $\varphi$ is $C^0$-close enough to the identity, then $\varphi(L')$ is included in a Weinstein neighbourhood of $L'$. This will contribute to the definition of the constants $\delta\text{ and }l\in\R$ at stake in \autoref{cor maj bottleneck} and \autoref{lem gamma C0}. By abuse of notation, we also denote $L',\varphi(L')$ their respective images in $T^*L'$ by a Weinstein embedding.  Denoting $\calD$ a perturbation data in $T^*L'$, we also denote $\calD$ its pull-back by the chosen Weinstein embedding. Let us recall that $HF^s(\varphi(L'),L';\calD,M)$ is the filtered cohomology computed in $M$ and $HF^s(\varphi(L'),L';\calD,T^*L')$ the filtered cohomology computed in $T^*L'$. We aim to prove that these two cohomologies are isomorphic and that this isomorphism respects the filtration.

This section is thus dedicated to the proof of the following \autoref{lem egalite hom cotangent}. The idea for this is to localize the Floer trajectories near $L'$. Indeed, this will imply that the Floer trajectories in $M$ and $T^*L'$ are in $1:1$ correspondence, and thus the two cochain complexes are isomorphic.

\begin{lemma}\label{lem egalite hom cotangent}
If $\varphi$ is $C^0$-close to the identity, then for an arbitrary choice of data there exists $C\in\R$ such that for all $s\in\R$ $$HF^s(\varphi(L'),L';\calD,M)\cong HF^s(\varphi(L'),L';\calD,T^*L')[C].$$
We can actually choose the primitive functions of the $1$-forms $\lambda_M$ and $\lambda_{T^*L'}$ restricted to the Lagrangian submanifolds such that the shift $C$ is equal to $0$.
\end{lemma}

\begin{proof}
The idea here is to retract the Lagrangian submanifold $\varphi(L')$ by the negative Liouville flow. This will decrease the diameter of the spectrum and thus allow to have small enough energy estimates on the moduli spaces.

 Let us choose two Weinstein's tubular neighbourhoods $U$ and $W$ of $L'$ such that $U\Subset W$.  We denote $\psi:W\rightarrow T^*L'$, the symplectic embedding provided by Weinstein's theorem. Let us suppose that $\varphi$ is close to $\Id_M$, so that we have $\varphi(L')$ included in $U$. We have two Liouville forms on $W$. The first one is the Liouville form $\lambda_M$ restricted to $W$. The second one is the Liouville form obtained from the Liouville form $\lambda_{T
^*L'}$ on $T^*L'$: $\psi^*\lambda_{T^*L'}$. Let us recall that $\psi^*\lambda_{T^*L'}-\lambda_M$ is closed  on $W$. Since $H^1(L',\R)=0$ we have $H^1(W,\R)=0$. Then $\psi^*\lambda_{T^*L'}-\lambda_M$ is exact on $W$ and consequently there exists a function $F:W\rightarrow\R$ such that $\psi^*\lambda_{T^*L'}=(\lambda_M)_{|W}+dF$.

Let us pick a cut-off function $\beta:W\rightarrow \R$ such that $\beta$  is constant, equal to $1$ on $U$ and equal to $0$ near the boundary of $W$. By abuse of notation, we denote $F$ the function defined on $M$ equal to $\beta F$ on $W$ and continuously extended by $0$ outside of $W$. The $1$-form $(\lambda_M+dF)$ is a Liouville form on $M$ equal to $\psi^*\lambda_{T^*L'}$ on $U$. We thus obtain a globally defined negative Liouville flow (i.e. the flow of the negative Liouville vector field) on $M$ which preserves $U$ and matches with the negative Liouville flow on $T^*L'$.

In $T^*L'$, let us denote $\varphi^t_{-\calL}$ the negative Liouville flow. When we apply this flow to $\psi(\varphi(L'))$ for $t\in\R^+$, we obtain a smooth path $(L'_t)_{t\in\R^+}$ of Lagrangian submanifolds in $T^*L'$. We can now consider the smooth path of Lagrangian submanifolds in $M$ given by $(L_t)_{t\in\R^+}=(\psi^{-1}(L'_t))_{t\in\R^+}$.

\begin{lemma}\label{lem égalité spectres}
For all $t\in\R^+$ we have $\Spec(L'_t,L';\mathcal{D}, T^*L')=\Spec(L_t,L';\mathcal{D}, M) + C_t$, where $C_t\in \R$. Moreover, we can choose the primitive functions of the $1$-forms $\lambda_M$ and $\lambda_{T^*L'}$ restricted to the Lagrangian submanifolds such that $C_t$ is equal to $0$ for all $t$.
\end{lemma}

From now on, assume that the primitives on the Lagrangian submanifolds are chosen so that for all $t\in\R$, $C_t=0$. Since in $T^*L'$ we have $(\varphi^t_{-\calL})^*\omega = e^{-t}\omega$, and $(\varphi^t_{-\calL})$ is equal to the identity on $L'$, we get
\begin{equation}\label{spectre flot de Liouville}
\Spec(L'_t,L';\mathcal{D}, T^*L')=e^{-t} \Spec(L'_0,L';\mathcal{D}, T^*L').    
\end{equation}

\begin{figure}[h]
    \centering
    \includegraphics[scale=1.5]{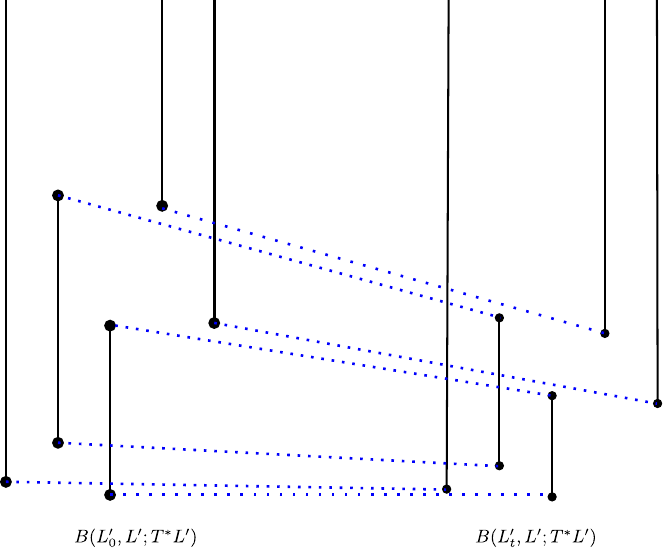}
    \caption{Evolution of the barcode during the Liouville retraction}
    \label{fig barcodes Liouville}
\end{figure}

\begin{lemma}\label{lem hom equ}
For $T$ large enough, there is a canonical identification between the cochain complexes
$CF(L'_T,L';\calD,T^*L')$ and $CF(L'_T,L';\calD,M)$ given by the Weinstein's neighbourhood embedding.
\end{lemma}
\begin{cor}
For $T$ large enough there is a canonical isomorphism
$$HF^s(L'_T,L';\calD,T^*L')\cong HF^s(L'_T,L';\calD,M)$$
holding for all $s\in\R$.
\end{cor}

\begin{rk}
Let us recall that if we are working in an exact symplectic manifold $(M,d\lambda)$ and the path $(L_t)_{t\in[0,1]}$ is a smooth path of exact Lagrangian submanifolds, there is a smooth path $(\phi_t)_{t\in[0,1]}$ in $\Ham(M,d\lambda)$ such that
\begin{equation}\label{chemin Lag Ham}
    \forall t\in[0,1],\quad \phi_t(L_0)=L_t.
\end{equation}
Consequently, since the paths $(L_t)_{t\in\R}$ and $(L'_t)_{t\in\R}$ are smooth, the associated barcode paths are continuous according to the previous expression (\ref{chemin Lag Ham}) and \autoref{prop barcode distance Hofer}. Let us denote $B_t=B(L_t,L';\calD,M)$ and $B'_t=B(L'_t,L';\calD,T^*L')$. The previous lemma tells that for $T$ large enough, $B_T=B'_T$. 

Moreover let $(B_t)_{t\in[0;1]}$ be a continuous path of barcodes such that there is a positive continuous function $f:\R\rightarrow\R$ which satisfies
$$\forall t\in[0;1],\Spec(B_t)=f(t)\Spec(B_0).$$
Since there is no bifurcation in the spectrum, $B_t$ is a dilation by $f(t)$ of $B_0$.
\end{rk}

\begin{lemma}\label{lem eg Liouville}
Let $L$ be a closed exact Lagrangian submanifold in a Weinstein neighbourhood $U$ of $L'$ with associated embedding $\psi$. For all $t\in[-T,0]$  let us denote $L'^t=\varphi_{-\calL}^{t}\circ \psi (L)$. Assume that for all $t\in[-T,0]$ $L'^t\subset U$. Let us denote $L^t=\psi^{-1}(L'^t)$. Then for all $s\in\R$
$$HF^s(L'^t,L';(\varphi_{-\calL}^{t})_*\calD,T^*L')\cong HF^{se^{-t}}(\psi(L),L';\calD,T^*L'),$$
$$HF^s(L^t,L';(\varphi_{-\calL}^{t})_*\calD,M)\cong HF^{se^{-t}}(L,L';\calD,M).$$
\end{lemma}
Applying this lemma to $L_T$ together with \autoref{lem hom equ} and the fact that $\varphi_{-\calL}^{-T}(L'_T)=\varphi(L')$, we finally obtain
$$HF^s(\varphi(L'),L';\calD,M)\cong HF^s(\varphi(L'),L';\calD,T^*L')$$
for all $s\in\R$. This concludes the proof of \autoref{lem egalite hom cotangent}.
\end{proof}

Let us now prove the Lemmas used in the proof of \autoref{lem egalite hom cotangent}.

\begin{proof}[Proof of \autoref{lem égalité spectres}]
\autoref{rk eg action intersection} tells us that the complexes
$CF(L'_t,L';\mathcal{D}, T^*L')$ and $CF(L_t,L';\mathcal{D}, M)$ can respectively be seen as the complexes $CF({L'}_t^{\mathcal{D}},{L'};T^*L')$ and \\ $CF(L_t^{\mathcal{D}},{L'}; M)$ where ${L'}_t^{\mathcal{D}}$ and $L_t^{\mathcal{D}}$ respectively denote the images of $L'_t$ and $L_t$ by the time-$1$ of the Hamiltonian perturbation as explained in \autoref{rk eg action intersection}. The actions of the original complexes and those of the new ones are equal up to a shift by constants respectively $c$ and $c'$. We assume, without any loss of generality, that we have chosen a good almost complex structure.

Fix $t\in \R^+$. Let $x$ be in $\chi(L_t^{\mathcal{D}},L')\subset M$, with action $\calA_{L_t^{\mathcal{D}},L'}(x)$. Denote $x'=\psi(x)$ which is consequently in $\chi({L'}_t^{\mathcal{D}},L')\subset T^*L'$ with action $\calA_{{L'}_t^{\mathcal{D}},L'}(x')$. Set $C_t=\calA_{{L'}_t^{\mathcal{D}},L'}(x')-\calA_{L_t^{\mathcal{D}},L'}(x)$.

For any other $y\in \chi(L_t^{\mathcal{D}}, L')$ together with $y'=\psi(y)\in \chi({L'}_t^{\mathcal{D}}, L')$, let us denote $\gamma_1$ a path from $x$ to $y$ in $L_t^{\mathcal{D}}$ and $\gamma_2$ a path from $y$ to $x$ in $L'$. We denote $\gamma_1'$ and $\gamma_2'$ their respective images by $\psi$. From \autoref{def action} and the fact that the differential decreases the action, we have
$$\calA_{L_t^{\mathcal{D}},L'}(y)-\calA_{L_t^{\mathcal{D}},L'}(x)=\int_{\gamma_1}\lambda_M+\int_{\gamma_2}\lambda_M,$$
$$\calA_{{L'}_t^{\mathcal{D}},L'}(y')-\calA_{{L'}_t^{\mathcal{D}},L'}(x')=\int_{\gamma'_1}\lambda_{T^*L'}+\int_{\gamma'_2}\lambda_{T^*L'}.$$
Denoting $\gamma_1\sharp\gamma_2$ the concatenation of $\gamma_1$ and $\gamma_2$ we get
\begin{eqnarray*}
\calA_{{L'}_t^{\mathcal{D}},L'}(y') & = & \calA_{{L'}_t^{\mathcal{D}},L'}(x') + \int_{\gamma_1'\sharp\gamma_2'}\lambda_{T^*L'}\\
& = & \calA_{{L'}_t^{\mathcal{D}},L'}(x') + \int_{\gamma_1\sharp\gamma_2}\psi^*\lambda_{T^*L'}\\
& = & \calA_{{L'}_t^{\mathcal{D}},L'}(x') + \int_{\gamma_1\sharp\gamma_2}\lambda_{M}+dF\\
& = & \calA_{{L'}_t^{\mathcal{D}},L'}(x') + \int_{\gamma_1\sharp\gamma_2}\lambda_{M}\quad\text{ since $\gamma_1\sharp\gamma_2$ is a loop}\\
& = & \calA_{{L'}_t^{\mathcal{D}},L'}(x') + \calA_{L_t^{\mathcal{D}},L'}(y)-\calA_{L_t^{\mathcal{D}},L'}(x) \\
& = & \calA_{L_t^{\mathcal{D}},L'}(y) + C_t.
\end{eqnarray*}
Since this is true for any $t\in \R^+$ and any pair $(y,y')$ such as before, we can conclude that
$$\forall t\in\R^+,\quad \Spec({L'}_t^{\mathcal{D}},L';T^*L')=\Spec(L_t^{\mathcal{D}},L';M) + C_t.$$
So
$$\forall t\in\R^+,\quad \Spec({L'}_t,L';\mathcal{D},T^*L')=\Spec(L_t,L';\mathcal{D},M) + C_t+c'-c.$$
Now, for all $t$, choosing two primitive functions on $L'_t$ and $L_t$ such that $\calA_{{L'}_t^{\mathcal{D}},L'}(x')-c=\calA_{{L}_t^{\mathcal{D}},L'}(x)-c'$ gives $C_t+c'-c=0$, which finishes this proof. 

\end{proof}

\begin{proof}[Proof of \autoref{lem hom equ}]
These two cochain complexes are generated by the perturbed intersection points, which are identified by Weinstein's neighbourhood embedding. To prove this lemma, we thus have to show that for $T$ large enough, the differential is the same, i.e. that the $J$-holomorphic curves between two intersection points agree. To do so we will show that if $T$ is large enough, no such $J$-holomorphic curves can go outside of $W$.

 Since we are working with a Liouville domain, which is always tame, Sikorav's proposition $4.3.1$ and its corollary in \cite{sik94} are verified. Consequently, there exists a constant $\kappa\in\R$, such that for any compact subset $K$, any compact connected $J$-holomorphic curve $u$ such that $u\cap K\neq\emptyset$, and $\partial u\subset K$ satisfies 
 \begin{equation}\label{inclusion Audin Lafontaine}
     u\subset U(K,\kappa \calA(u)),
 \end{equation}
 where $U(K,\kappa \calA(u))$ is the $\kappa \calA(u)$-neighbourhood of $K$.
 Let us fix $\delta>0$ small enough such that we can find a compact neighbourhood $K$ of $L'$ such that $ U(K,\delta)\subset W$.

Let us denote $\Gamma_t$, the diameter of the spectrum $\Spec(L'_t,L';\mathcal{D}, T^*L')$, which is equal by \autoref{lem égalité spectres} to the diameter of the spectrum $\Spec(L'_t,L';\mathcal{D}, M)$. From the previous equality \ref{spectre flot de Liouville}, we have $\Gamma_t=e^{-t}\Gamma_0$. Set $t_\delta=\ln(\frac{\Gamma_0\kappa}{\delta})$. We then have
$$\forall t \geq t_\delta, \Gamma_t\leq \frac{\delta}{\kappa}.$$
Let us recall that, the area of a $J$-holomorphic strip between two intersection points is equal to the difference of action between the two intersection points. This area is thus bounded by the diameter of the spectrum $\Gamma_t$. Let us fix $T>t_\delta$. A $J$-holomorphic strip $u$ between two generators of $CF(L'_T,L';\calD,M)$ satisfies $\calA(u)\leq \frac{\delta}{\kappa}$. Inclusion \ref{inclusion Audin Lafontaine} then becomes
$$u\subset U(K,\kappa \calA(u))= U(K,\delta)\subset W.$$
This means that the $J$-holomorphic strips defining the differential of the chain complex $CF(L'_T,L';\calD,M)$ stay in $W$. They are identified by the embedding $\psi$ to the $J$-holomorphic strips defining the differential of the chain complex $CF(L'_T,L';\calD,T^*L')$. Consequently the differential of the two chain complexes behave well with respect to the embedding $\psi$. This concludes the proof of this lemma.
\end{proof}

\begin{rk}
In this proof, we only dealt with $J$-holomorphic curves computing the differential. However, we can conduct the exact same proof with other moduli spaces. This implies that the $\mu^k$-operations are preserved by the isomorphism given by \autoref{lem hom equ}.
\end{rk}

\begin{proof}[Proof of \autoref{lem eg Liouville}]
The symplectic invariance of Floer cohomology tells that there is a function $f:\R\rightarrow\R$ such that
$$HF^s(L'^t,\varphi_{-\calL}^{t}(L');(\varphi_{-\calL}^{-t})_*\calD,T^*L')\cong HF^{f(s)}(\psi(L),L';\calD,T^*L'),$$
$$HF^s(L^t,\psi^{-1}\circ\varphi_{-\calL}^{t}(L');(\varphi_{-\calL}^{t})_*\calD,M)\cong HF^{f(s)}(L,L';\calD,M).$$
We can indeed write the second isomorphism since the negative Liouville flow has been globally defined on $M$.
Since $\varphi_{-\calL}^t(L')=L'$ for all $t$, we have
$$HF^s(L'^t,L';(\varphi_{-\calL}^{t})_*\calD,T^*L')\cong HF^{f(s)}(\psi(L),L';\calD,T^*L'),$$
$$HF^s(L^t,L';(\varphi_{-\calL}^{t})_*\calD,M)\cong HF^{f(s)}(L,L';\calD,M).$$
Moreover Equality \ref{spectre flot de Liouville} tells that $f(s)=se^{-t}$. We then have
$$HF^s(L'^t,L';(\varphi_{-\calL}^{t})_*\calD,T^*L')\cong HF^{se^{-t}}(\psi(L),L';\calD,T^*L').$$
The same computation as in \autoref{lem égalité spectres} gives
$$HF^s(L^t,L';(\varphi_{-\calL}^{t})_*\calD,M)\cong HF^{se^{-t}}(L,L';\calD,M).$$
\end{proof}

Let us assume, once and for all that $\varphi$ is sufficiently close to identity, so that $\varphi(L')$ is inside the Weinstein neighbourhood of $L'$ and reciprocally.

\subsection{Bounding the bottleneck distance by the spectral norm}\label{sec bottleneck}
In this section, we will bound the bottleneck distance by the spectral norm $\gamma$. \autoref{cor maj bottleneck} is a consequence of the following proposition.
\begin{prop}\label{prop majoration par bottleneck}
Let $L$ and $L'$ be two closed exact Lagrangian submanifolds in a Liouville domain $(M,\omega)$. There exists $\delta>0$, independant of $L$, such that for all  $\varphi\in\Symp(M,\omega)$ satisfying $d_{C^0}(\varphi,\Id_M)\leq\delta$, then there exists $C\in\R$ such that
$$d_{bottle}(B(L',L),B(\varphi(L'),L)[C])\leq \tfrac{1}{2}\gamma(L',\varphi(L')).$$
\end{prop}
In \cite{KS18}, Kislev and Shelukhin proved a similar statement in a different setting. In their case, $L=L'$ is a weakly monotone Lagrangian submanifold in a closed symplectic manifold and $\varphi$ is a Hamiltonian diffeomorphism. The following proof of \autoref{prop majoration par bottleneck} is an adaptation of their proof to our setting.

We choose $\delta>0$ so that, for all $\varphi\in\Symp(M,\omega)$, if $d_{C^0}(\varphi,\Id_M)\leq\delta$ then $\varphi(L')$ is included in a Weinstein neighbourhood of $L'$. We will denote this Weinstein neighbourhood $W(L')$.

We can now prove \autoref{cor maj bottleneck} required to prove \autoref{prop contL barcode}.

\begin{proof}[Proof of \autoref{cor maj bottleneck}]
To prove this proposition, we will apply \autoref{prop majoration par bottleneck} to the symplectomorphism $\psi^{-1}\circ\varphi$. As in \autoref{prop majoration par bottleneck}, we choose $\delta>0$ so that, for all $\phi\in\Symp(M,\omega)$, if $d_{C^0}(\phi,\Id_M)\leq\delta$ then $\phi(L')$ is included in $W(L')$, a Weinstein neighbourhood of $L'$. Let us assume that $d_{C^0}(\varphi,\psi)\leq\delta$.
\begin{eqnarray*}
d_{C^0}(\varphi,\psi)&=& \max\left\{\sup\limits_{x\in M}d(\varphi(x),\psi(x)),\quad\sup\limits_{x\in M}d(\varphi^{-1}(x),\psi^{-1}(x))\right\}\\
&\geq& \sup\limits_{x\in M}d(\varphi^{-1}(x),\psi^{-1}(x))\\
&=& \sup\limits_{x\in M}d(\psi^{-1}\circ\varphi(x),x)\\
&=& d_{C^0}(\psi^{-1}\circ\varphi,\Id_M).
\end{eqnarray*}
So we get $$d_{C^0}(\psi^{-1}\varphi,\Id_M)\leq\delta.$$
We introduced the set of barcodes quotiented by an overall shift $\hat{\calB}$ to get rid of the shift in the inequality of \autoref{prop majoration par bottleneck}.  Indeed, when working with the barcodes in $\hat{\calB}$, this inequality becomes
$$d_{bottle}(\hat{B}(L',L),\hat{B}(\varphi(L'),L))\leq \tfrac{1}{2}\gamma(L',\varphi(L')).$$
By invariance of the barcode under the action of a symplectomorphism, we have
$$d_{bottle}(\hat{B}(\varphi(L'),L),\hat{B}(\psi(L'),L))=d_{bottle}(\hat{B}(L',\psi^{-1}(L)),\hat{B}(\psi^{-1}\circ\varphi(L'),\psi^{-1}(L))).$$
By the previous inequality and \autoref{prop majoration par bottleneck}, we then have
$$d_{bottle}(\hat{B}(L',\psi^{-1}(L)),\hat{B}(\psi^{-1}\circ\varphi(L'),\psi^{-1}(L)))\leq \tfrac{1}{2}\gamma(L',\psi^{-1}\circ\varphi(L')),$$
which concludes the proof of this proposition.

\end{proof}

Let us now prove \autoref{prop majoration par bottleneck} and the desired bound. We start by introducing the interleaving maps.

\bigskip

\noindent\textbf{Set up to define the interleaving maps}

As explained in \autoref{rk exacteness}, the condition $H^1(L',\R)=0$ guarantees that for all $\varphi\in \Symp(M,\omega)$, $\varphi(L')$ is an exact Lagrangian submanifold. Hence, we can apply Abouzaid-Kragh's \autoref{thm NLC} \cite{AK18}, thus obtaining two isomorphisms
$$ HF(L',L';\calD,T^*L')\xrightarrow[\sigma]{\sim} HF(L',\varphi(L');\calD,T^*L'),$$
$$HF(\varphi(L'),\varphi(L');\calD,T^*L')\xrightarrow[\sigma']{\sim} HF(\varphi(L'),L';\calD,T^*L').$$
Moreover, by \autoref{prop lien entre HF et H} applied to $L'$, together Poincaré duality there is an isomorphism $$\theta:H_*(L')\rightarrow HF^{n-*}(L',L';\calD,T^*L').$$
By symplectic invariance of Floer cohomology, we have $$HF^*(\varphi(L'),\varphi(L');\phi^*\calD,T^*L')\cong HF^*(L',L';\calD,T^*L').$$
As above, we also have an isomorphism
$$\theta':H_*(L')\rightarrow HF^{n-*}(\varphi(L'),\varphi(L');\calD,T^*L').$$

Let us choose $c\in HF(L',L';\calD,T^*L')$ to be the class $\theta([L'])$, and $c'\in HF(\varphi(L'),\varphi(L');\calD,T^*L')$ the class $\theta'([L'])$. Moreover assume that the gradings are chosen so that $c$ and $c'$ are both of degree $0$.

\autoref{lem egalite hom cotangent} provides two isomorphisms $\zeta$ and $\zeta'$ between $HF(L',\varphi(L');T^*L')$ and $HF(L',\varphi(L');M)$ and between $HF(\varphi(L'),L';T^*L')$ and $HF(\varphi(L'),L';M)$.
We can now choose two cycles $x\in CF(L',\varphi(L');M)$ and $y\in CF(\varphi(L'),L';M)$ such that
$$[x]=\zeta(\sigma(c))$$
$$[y]=\zeta'(\sigma'(c')).$$

Let us choose two primitive functions $f':L'\rightarrow \R$ and $g:\varphi(L')\rightarrow \R$ such that $df'=\lambda_{|L'}$, $dg=\lambda_{|\varphi(L')}$ and such that we can find
\begin{itemize}
    \item $z$ such that $[z]=c\in HF(L',L';\calD)$ with $\calA(z)\leq \varepsilon'/2\ll\varepsilon_0$
    \item $z'$ such that $[z']=c'\in HF(\varphi(L'),\varphi(L');\calD)$ with $\calA(z')\leq \varepsilon'/2\ll\varepsilon_0.$
\end{itemize}

According to the previous choices of degree, we actually have $[z]\in HF^0(L',L';\calD)$ and $[z']\in HF^0(\varphi(L'),\varphi(L');\calD)$.

\begin{lemma}\label{lem mult identité}
The multiplication maps
$$m_2(\cdot,z):CF^*(L',L;\calD)\rightarrow CF^*(L',L;\calD)[\varepsilon']$$
$$m_2(\cdot,z'):CF^*(\varphi(L'),L;\calD)\rightarrow CF^*(\varphi(L'),L;\calD)[\varepsilon']$$
are filtered chain-homotopic to the standard inclusions and hence induce the $\varepsilon'$-shift maps on the persistence modules.
\end{lemma}

\begin{proof}
This lemma is an immediate consequence of \autoref{lem mult id barcodes}.
\end{proof}

By abuse of notation, to make the following expressions clearer, we denote $[L']=\zeta\circ\sigma\circ\theta ([L'])\in HF^0(L',\varphi(L');\calD,M)$ and $[\varphi(L')]=\zeta'\circ\sigma'\circ\theta'([L'])\in HF^0(\varphi(L'),L';\calD,M)$.
Now, let us choose $x\in CF^0(L',\varphi(L');\calD)$ and $y\in CF^0(\varphi(L'),L';\calD)$ as above such that:
$$l([L'];L',\varphi(L');\calD) \leq \calA(x)=a \leq l([L'];L',\varphi(L');\calD)+\varepsilon',$$
$$l([\varphi(L')];\varphi(L'),L';\calD) \leq \calA(y)=b \leq l([\varphi(L')];\varphi(L'),L';\calD) + \varepsilon',$$
which is possible by definition of $l([L'];L',\varphi(L');\calD)$ and $l([\varphi(L')];\varphi(L'),L';\calD)$.

Moreover, by definition of $x,y$, we have $[\mu^2(y,x)]=[z]\in HF^0(L',L';\calD)$. Indeed, up to the appropriate isomorphisms, the cycles $x,y,z$ all represent the same class $[L]$ in their respective cochain complexes. With our particular choice of perturbation data for the pair $(L',L')$ as explained in \autoref{rk choix perturbation data}, the cycle $z$ is the only representative of his class. The same argument holds for $z'$. Consequently, we have the following lemma.

\begin{lemma}
$$\mu^2(y,x)=z\in CF^0(L',L';\calD),$$
$$\mu^2(x,y)=z'\in CF^0(\varphi(L'),\varphi(L');\calD).$$
\end{lemma}

\begin{rk}\label{rk sans la NLC}
If we choose to work with $\varphi$ being a Hamiltonian diffeomorphism and not only a symplectomorphism, the definition of $x$ and $y$ is much easier. In this case, it is achieved without Abouzaid-Kragh's result \cite{AK18} of \autoref{thm NLC}.

Indeed, continuation morphisms give the isomorphisms:
$$HF^*(\varphi(L'),L')\cong HF^*(L',L')\cong HF^*(\varphi(L'),\varphi(L'))\cong HF^*(L',\varphi(L')).$$

Since these continuations morphisms are compatible with the product structure on Lagrangian Floer cohomology, we can directly define $x$ and $y$, and it is easy to see that the product by these elements will not be constant equal to $0$. Indeed we easily have
$$[\mu^2(y,x)]=[z]$$
$$[\mu^2(x,y)]=[z'].$$
Moreover, the two multiplication operators $m_2(\cdot,z)$ and $m_2(\cdot,z')$ are still filtered chain-homotopic to the standard inclusion.
\end{rk}

~~\\
\textbf{Bounding the bottleneck distance}

Now that our objects are defined, we can adapt the Kislev-Shelukhin method \cite{KS18} to our context. Except for the context, we do not claim anything new here. The point here is to carefully study the shifts of action induced by the the multiplication by the elements introduced above. Let us start with the two following lemmas.

\begin{lemma}
The maps
\begin{eqnarray*}
    \mu^2(\cdot,x):CF^*(\varphi(L'),L;\calD)\rightarrow CF^*(L',L;\calD)[a+\varepsilon']
    \\
    \mu^2(\cdot,y):CF^*(L',L;\calD)\rightarrow CF^*(\varphi(L'),L;\calD)[b+\varepsilon']
\end{eqnarray*}
are well-defined and induce filtered maps of chain complexes.
\end{lemma}

\begin{lemma}
The maps
\begin{eqnarray*}
    \mu^2(\mu^2(\cdot,y),x):CF^*(L',L;\calD)\rightarrow CF^*(L',L;\calD))[a+b+3\varepsilon']
    \\
    \mu^2(\mu^2(\cdot,x),y):CF^*(\varphi(L'),L;\calD)\rightarrow CF^*(\varphi(L'),L;\calD)[a+b+3\varepsilon']
\end{eqnarray*}
are well-defined and filtered chain homotopic to the multiplication operators:
\begin{eqnarray*}
    \mu^2(\cdot,\mu^2(y,x)):CF^*(L',L;\calD)\rightarrow CF^*(L',L;\calD)[a+b+3\varepsilon']
    \\
    \mu^2(\cdot,\mu^2(x,y)):CF^*(\varphi(L'),L;\calD)\rightarrow CF^*(\varphi(L'),L;\calD)[a+b+3\varepsilon']
\end{eqnarray*}

\end{lemma}
\begin{proof}
These two lemmas directly follow from the discussion on the product structure. \autoref{lem mult pers mod} gives the first one and \autoref{lem associativite barcodes} the second one.
\end{proof}

\begin{rk}
In Kislev and Shelukhin's paper \cite{KS18}, there is another term in the previous equality which is a boundary. This additional term induces a shift in action by a constant $\beta$ which vanishes in our case.
\end{rk}

We now have the relation between the previous maps and the multiplication by $z$ or $z'$: the maps
\begin{eqnarray*}
    \mu^2(\cdot,\mu^2(y,x)):CF^*(L',L;\calD)\rightarrow CF^*(L',L';\calD)[a+b+3\varepsilon']
    \\
    \mu^2(\cdot,\mu^2(x,y)):CF^*(\varphi(L'),L;\calD)\rightarrow CF^*(\varphi(L'),L;\calD)[a+b+3\varepsilon']
\end{eqnarray*}
are equal to the multiplication operators:
\begin{eqnarray*}
        \mu^2(\cdot,z):CF^*(L',L;\calD)\rightarrow CF^*(L',L;\calD)a+b+3\varepsilon']
    \\    \mu^2(\cdot,z'):CF^*(\varphi(L'),L;\calD)\rightarrow CF^*(\varphi(L'),L;\calD)[a+b+3\varepsilon']
    \end{eqnarray*}

Following \autoref{lem mult identité}, we obtain on the level of filtered homology the shifts of persistence modules morphisms
\begin{eqnarray*}
        sh_{L'}:CF^*(L',L;\calD)\rightarrow CF^*(L',L;\calD)[a+b+4\varepsilon']
    \\  sh_{\varphi(L')}:CF^*(\varphi(L'),L;\calD)\rightarrow CF^*(\varphi(L'),L;\calD)[a+b+4\varepsilon'].
\end{eqnarray*}
Let us recall that the barcodes are $C^2$-continuous by \autoref{prop barcode distance Hofer}. We can take the limit as the Hamiltonian part of the perturbation data goes to zero as explained after \autoref{prop barcode distance Hofer} and assume that
$$a < l([L'];L',\varphi(L'))+2\varepsilon',$$
$$b < l([L'];\varphi(L'),L')+2\varepsilon'.$$
Consequently we have shift maps of barcodes without the perturbation data:
\begin{eqnarray*}
        sh_{L'}=\mu^2(\cdot,x)\circ\mu^2(\cdot,y):B(L',L)\rightarrow B(L',L)[\gamma(L',\varphi(L'))+6\varepsilon']
    \\  sh_{\varphi(L')}=\mu^2(\cdot,y)\circ\mu^2(\cdot,x):B(\varphi(L'),L)\rightarrow B(\varphi(L'),L)[\gamma(L',\varphi(L'))+6\varepsilon'].
\end{eqnarray*}
Indeed, as discussed in Section \ref{sec gamma}, $$l([L'];L',\varphi(L'))+l([L'];\varphi(L'),L')=\gamma(L',\varphi(L'))\geq 0.$$
For readability reasons, we denote
$$\alpha=l([L'];L',\varphi(L'))\text{ and }\bar{\alpha}=l([L'];\varphi(L'),L').$$
With this expression, the multiplication operators appear as maps between persistence modules:
\begin{eqnarray*}
    \mu^2(\cdot,x):B(\varphi(L'),L)\rightarrow B(L',L)[\alpha+3\varepsilon']
    \\
    \mu^2(\cdot,y):B(L',L)\rightarrow B(\varphi(L'),L)[\bar{\alpha}+3\varepsilon'].
\end{eqnarray*}
Let us recall that, by \autoref{lem def gamma}, we have $\gamma(L',\varphi(L'))=\alpha+\bar{\alpha}$. Consequently, the previous multiplication operators can be written as
\begin{eqnarray*}
    \mu^2(\cdot,x):B(\varphi(L'),L)\rightarrow B(L',L)[\tfrac{1}{2}(\alpha-\bar{\alpha})][\tfrac{1}{2}\gamma(L',\varphi(L'))+3\varepsilon']
    \\
    \mu^2(\cdot,y):B(L',L)[\tfrac{1}{2}(\alpha-\bar{\alpha})]\rightarrow B(\varphi(L'),L)[\tfrac{1}{2}\gamma(L',\varphi(L'))+3\varepsilon']
\end{eqnarray*}
Together with the previous identity of persistence modules, this is the exact definition of the fact that $B(L',L)$ and $B(\varphi(L'),L)[\frac{1}{2}(\alpha-\bar{\alpha})]$ are $\frac{1}{2}\gamma(L',\varphi(L'))+3\varepsilon'$-interleaved. Taking the limit as $\varepsilon'$ goes to zero, we get
\begin{equation}\label{eq majoration bottleneck}
d_{bottle}(B(L',L),B(\varphi(L'),L)[\tfrac{1}{2}(\alpha-\bar{\alpha})])\leq \tfrac{1}{2}\gamma(L',\varphi(L')).    
\end{equation}
Setting $C=\tfrac{1}{2}(\alpha-\bar{\alpha})$, this concludes the proof of \autoref{prop majoration par bottleneck}.
    
\subsection{Bounding the spectral norm by the $C^0$-distance}\label{sec gamma C0}

We will now prove \autoref{lem gamma C0}. This proof is an adaptation to our context of a lemma and a proof of Buhovsky-Humilière-Seyfaddini \cite{BHS19}. In their paper, they proved the same result for a Lagrangian submanifold Hamiltonian isotopic to the zero section in a cotangent bundle.

Here we are working with $L'$ being a closed exact Lagrangian submanifold in $M$. Let us denote $W(L')$ a Weinstein neighbourhood of $L'$. By definition, if $\varphi\in \Symp(M,\omega)$ is close enough to $\Id_M$, then $\varphi(L')\subset W(L')$. By abuse of notation, we also respectively denote $B$ and $\varphi(L')$ the images by a Weinstein embedding of respectively $B$ and $\varphi(L')$ in $T^*L'$, where $B$ is a ball in $W(L')$.

We will start by stating two key lemmas without any proof. Indeed these lemmas are adaptations to our particular context of Buhovsky-Humilière-Seyfaddini' s lemmas and the proofs they give apply verbatim to our situation. We will then apply these lemmas and do some basic computations to prove \autoref{lem gamma C0}. For more details, one can also refer to the author's thesis.

\begin{lemma}\label{lem Lip BHS avec B}
Let $B$ be a ball in $L'$. Let $\Symp_B(M,\omega):=\{\varphi\in \Symp(M,\omega)|\quad \varphi(L')\cap T^*B= 0_B)\}$. There exist $\delta>0$ and $C>0$ such that for any $\varphi\in \Symp_B(M,\omega)$, if $d_{C^0}(\Id_M,\varphi)\leq \delta$, then $\gamma(L',\varphi(L'))\leq C d_{C^0}(\Id_M,\varphi)$.
\end{lemma}

In order to finish the proof of \autoref{lem gamma C0}, we need to reduce to \autoref{lem Lip BHS avec B}. Indeed in the hypothesis, we do not have such a ball $B$. To do so, we will use and adapt to our context a trick from \cite{BHS19}. This trick consists in doubling the coordinates and introducing the following auxiliary map:
$$\Phi:\varphi\times\varphi^{-1}=M\times M\rightarrow M\times M,$$
where $M\times M$ is equipped with the natural symplectic form $\omega\oplus\omega$.

Buhovsky-Humilière-Seyfaddini \cite{BHS19} also gives the following lemma:

\begin{lemma}\label{lem B BHS}
For any ball $B$ in $M$, there is a smaller ball $B'\subset B$ with the following property. There exists $\Delta>0$ such that for any $\varphi\in \Symp(M,\omega)$ with $d_{C^0}(\varphi, \Id_M)<\Delta$, we can find a symplectomorphism $\Psi\in \Symp(M\times M,\omega\oplus\omega)$ satisfying:
\begin{enumerate}
\item $supp(\Psi)\subset B\times B$ and $supp(\Phi\circ\Psi)\subset M\times M\backslash B'\times B'$,
\item $d_{C^0}(\Psi, \Id_{M\times M})<C_B d_{C^0}(\varphi, \Id_M)$ and $d_{C^0}(\Phi\circ\Psi, \Id_{M\times M})<C'_B d_{C^0}(\varphi, \Id_M)$, where $C_B$ and $C'_B$ do not depend on $\varphi$.
\end{enumerate}
\end{lemma}

This \autoref{lem B BHS} together with \autoref{lem Lip BHS avec B} will allow to conclude the proof of \autoref{lem gamma C0} and thus \autoref{prop contL barcode}. Indeed, we have proven that $B(L',L)$ and $B(\varphi(L'),L)[\tfrac{1}{2}(\alpha-\bar{\alpha})]$ are $\frac{1}{2}\gamma(L',\varphi(L'))$-interleaved. We now just have to find two constants $\kappa>0$ and $l>0$ such that if $d_{C^0}(\varphi, \Id_M) \leq l$, then $\gamma(L',\varphi(L'))\leq \kappa d_{C^0}(\varphi, \Id_M)$.

Let us pick a point $x\in L'$ and a ball $B_x$ centered on $x$. \autoref{lem B BHS} provides a smaller ball $B'$ also centered on $x$. Pick a smaller ball $B_0$, centered on $x$ and whose closure is included in $B'$. The same lemma also provides $l_0>0$ and a symplectomorphism $\Psi$ such that $l_0<\Delta$ and if $d_{C^0}(\varphi,\Id_M)<l_0$, then $\Phi\circ\Psi(L'\cap B_0\times L'\cap B_0)\cap T^*(B'\times B')= L'\cap B_0\times L'\cap B_0$.

Then, let us pick another ball $B_1$ centered on $y\in L'\times L'$ whose closure is included in $M\backslash B\times B$. Since $d_{C^0}(\Psi, \Id_{M\times M})\leq C_B d_{C^0}(\varphi, \Id_M)$ and $supp(\Psi)\subset B\times B$, we can find $l_1>0$ such that if $d_{C^0}(\varphi, \Id_M)< l_1$, then $\Psi$ and $B_1$ satisfy the conditions of \autoref{lem B BHS}.

Let us choose $l>0=\min\{\delta,l_0,l_1\}$. Then we have, using successively \autoref{prop formule de Kunneth} and its consequence of Equality (\ref{eq gamma kunneth}) and the triangle Inequality (\ref{eq gamma triang}) and the symmetry of $\gamma$ (\ref{eq sym gamma}), for all $\varphi$ such that $d_{C^0}(\varphi,\Id_M)<l$:
\begin{eqnarray*}
\gamma(L',\varphi(L')) & = &\frac{1}{2} \gamma(L'\times L',\Phi(L'\times L'))\\
& = & \frac{1}{2}\gamma(\Psi^{-1}\Phi^{-1}(L'\times L'),\Psi^{-1}(L'\times L'))\\ 
& \leq & \frac{1}{2}\gamma(L'\times L',\Psi^{-1}(L'\times L'))+ \frac{1}{2}\gamma(\Psi^{-1}\Phi^{-1}(L'\times L'),L'\times L')\\
& = & \frac{1}{2}\gamma(L'\times L',\Psi^{-1}(L'\times L')) + \frac{1}{2}\gamma(L'\times L', \Phi\Psi(L',L')).
\end{eqnarray*}
For the second equality, the same argument as in \autoref{lem égalité spectres} indeed tells that $\gamma(L'\times L',\Phi(L'\times L'))=\gamma(\Psi^{-1}\Phi^{-1}(L'\times L'),\Psi^{-1}(L'\times L'))$, when composing by $\Psi^{-1}\Phi^{-1}$. A similar argument holds for the first equality and for the last one.

Choosing $B_0\times B_0$ for the ball in \autoref{lem Lip BHS avec B}, we can apply it to $\Phi\circ\Psi$ for all $\varphi$ such that $d_{C^0}(\varphi,\Id_M)<l$. We then get that there is a constant $C_0>0$ such that for all these $\varphi$, we have:
\begin{eqnarray*}
\gamma(L'\times L',\Phi\circ\Psi(L'\times L')) & \leq & C_0 d_{C^0}(\Phi\circ\Psi,\Id_{M\times M})\\
& \leq & C_0 C'_B d_{C^0}(\varphi,\Id_M).
\end{eqnarray*}
Moreover, for all such $\varphi$, \autoref{lem Lip BHS avec B} gives for $\Psi$ a constant $C_1$:
\begin{eqnarray*}
\gamma(L'\times L',\Psi^{-1}(L'\times L')) & \leq & C_1 d_{C^0}(\Psi^{-1},\Id_{M\times M})\\
& \leq & C_1 C_B d_{C^0}(\varphi,\Id_M).
\end{eqnarray*}
Putting all this together, we get:
$$\gamma(L',\varphi(L')) \leq \frac{1}{2}(C_0 C'_B+C_1 C_B) d_{C^0}(\varphi,\Id_M).$$
By setting $\kappa=\frac{1}{2}(C_0 C'_B+C_1 C_B)$, we get that for all $\varphi$ such that $d_{C^0}(\varphi, \Id_M) \leq l$, then $\gamma(L',\varphi(L'))\leq \kappa d_{C^0}(\varphi, \Id_M)$.

\bigskip

Taking into account the discussions in Subsection \ref{subsec idea proof}, the proof of \autoref{prop contL barcode} is now complete, and this Section \ref{sec barcode continu} finished.

\section{The Dehn-Seidel twist in $C^0$-symplectic geometry}\label{sec Dehn twist}

Now that all our main tools have been introduced, we can prove our theorems regarding the Dehn-Seidel twist.

\subsection{Seidel's theorem}\label{subsec thm Seidel}
Since the square of the Dehn-Seidel twist has been proved to be isotopic to the identity in $\Diff(M^{2n})$ when $n=2$ \cite{SeiPhD}, it is a natural question to ask whether this is true in higher dimensions. Since this map is symplectic, it is also natural to wonder whether this also holds in $\Symp(M,\omega)$, or whether this is a purely smooth (non-symplectic) result. Even if the answer to the first question is still unknown, regarding the second one, Seidel proved in \cite{Seidel00} a stronger result, by considering images of Lagrangian submanifolds instead of directly considering the Dehn twist. For the reader's convenience, let us recall Seidel's theorem.

\begin{thm}[Seidel \cite{Seidel00}]\label{thm Seidel Dtwist}
Let $(M^{2n},\omega)$ be a compact symplectic manifold with contact type boundary, with n even, which satisfies $[\omega]=0$ and $2c_1(M,\omega)=0$. Assume that $M$ contains an $A_3$-configuration $(l_{\infty},l',l)$ of Lagrangian spheres.  Then $M$ contains infinitely many symplectically knotted Lagrangian spheres. More precisely, if one defines $L'^{(k)}=\tau_l^{2k}(L')$ for $k\in \Z$, then all the $L'^{(k)}$ are isotopic as smooth submanifolds of M, but no two of them are isotopic as Lagrangian submanifolds.
\end{thm}

Since no two of these Lagrangian submanifolds are isotopic as Lagrangian submanifolds, the following corollary is immediate.
\begin{cor}
$\tau_l^2$ is not in the identity component of $\Symp_c(T^*S^n)$, the compactly supported symplectomorphisms of $T^*S^n$.
\end{cor}
Indeed, if this symplectomorphismmorphism was in the identity component of $\Symp_c(T^*S^n)$, its conjugation by the embedding $j$ would also be in the identity component of $\Symp_c(M,\omega)$, and thus, all the Lagrangian spheres in \autoref{thm Seidel Dtwist} would be isotopic as Lagrangian submanifolds.

\begin{rk}\label{prop exactitude image Dtwist}
Let $L$ be a Lagrangian sphere in $M$. Then for any Lagrangian sphere $L'$ in $M$, $\tau_l(L')$ is a Lagrangian sphere as well.

It can be checked that the Dehn-Seidel twist is in fact an exact symplectomorphism.
\end{rk}

The proof of this theorem deeply relies on the isotopy invariance of Floer homology together with the action of the Dehn-Seidel twist on the Maslov index. The proof we will give for the analogous result in $C^0$ symplectic topology also relies on barcodes and consequently on Floer cohomology. However there are some technical difficulties to adapt Seidel's proof to our context.

\begin{rk}
A similar result holds when working with odd $n$. However, one should not consider the square of the Dehn-Seidel twist, but the cube of the composition of two Dehn-Seidel twists defined along different but intersecting Lagrangian submanifolds \cite{Seidel00}.
\end{rk}

We introduced the notion of Milnor fibres after \autoref{def Ak} as these are examples of manifolds satisfying the conditions required for \autoref{thm Seidel Dtwist}.

We state here the following technical lemma, which was a key argument in Seidel's proof \cite{Seidel00} and which will be useful in the following computations.

\begin{lemma}\label{prop graded Dt}
There is a unique $\calL^\infty$ grading $\tilde{\tau}_l$ of $\tau_l$ which acts trivially on the part of $\calL^\infty$ which lies over $M\setminus\im(i)$. It satisfies $\tilde{\tau_l}{\tilde{L}}=\tilde{L}[1-n]$ for any grading $\tilde{L}$ of $L$.
\end{lemma}
We recall that here $i$ is the Weinstein embedding at stake in the definition of the Dehn-Seidel twist as presented in the Introduction.

\bigskip

Let us briefly recall the property of this grading we will need for the following section. For more detailled explanations, see \cite{Seidel00}.

Denoting $\tilde{L}[k]$ the graded Lagrangian $\tilde{L}$ whose grading has been shifted by $k\in\Z/N$, we have the following useful property, where $\tilde{L}_0$ and $\tilde{L}_1$ are two $\Z/N$-graded Lagrangian submanifolds in a symplectic manifold $(M,\omega)$.
\begin{equation}\label{eg hom absolute grading}
HF^*(\tilde{L}_0[k],\tilde{L}_1[l])\cong HF^{*-k+l}(\tilde{L}_0[k],\tilde{L}_1[l]).    
\end{equation}
In addition we have the invariance under the action of a graded symplectomorphisms $\tilde{\varphi}$:
\begin{equation}\label{invariance symplecto graded}
HF^*(\tilde{\varphi}(\tilde{L}_0),\tilde{\varphi}(\tilde{L}_1))\cong HF^{*}(\tilde{L}_0,\tilde{L}_1),    
\end{equation}
the Poincaré duality
\begin{equation}\label{PD graded}
HF^*(\tilde{L}_1,\tilde{L}_0)\cong HF^{n-*}(\tilde{L}_0,\tilde{L}_1).    
\end{equation}
Finally, when \autoref{prop lien entre HF et H} holds, we have a graded counterpart:
\begin{equation}\label{HF H graded}
    HF^*(\tilde{L},\tilde{L})\cong \underset{i\in\Z}\bigoplus H^{*+iN}(L,\Z/2).
\end{equation}

\subsection{Exact triangle in Floer cohomology}\label{sec Long exact sequ}
As mentioned in the previous section, Floer cohomology will be essential to our proof. It is actually possible to compute the action of the Dehn-Seidel twist on the Floer cohomology of certain exact Lagrangian submanifolds. It is the object of the following theorem, also proved by Seidel \cite{Seidel01}.

\begin{thm}\label{thm long ex seq Seidel}
Let $l:S^n\rightarrow M$ be a Lagrangian sphere in $(M^{2n},\omega)$ with image $L$. For any two exact Lagrangian submanifolds $L_0,L_1\in M$, there is a long exact sequence of Floer cohomology groups:
$$
\xymatrix{
  HF(\tau_l(L_0),L_1) \ar[rr]^0& &HF(L_0,L_1)\ar[dl]^n\\
   &HF(L,L_1)\otimes HF(L_0,L) \ar[ul]^{1-n}&
}
$$
\end{thm}

Now that this theorem is stated, with Keating's work \cite{Kea12}, we make some computations of this long exact sequence, in order to use it in our context.
\begin{prop}\label{prop rk cohom Dehn twist}
Let $(M^{2n},\omega)$ be a connected Liouville domain, $n$ even, $2c_1(M,\omega)=0$. Assume that M contains an $A_2$-configuration of Lagrangian spheres $(l,l')$.
Then, for all $ m\in\N^*$, 
$$\mathrm{rk}(HF^*(\tau_l^{2m}(L'),L'))=2m.$$
\end{prop}

\begin{proof}
This readily follows from Keating's Proposition $6.4$ in \cite{Kea12} coming from Seidel's \autoref{thm long ex seq Seidel} in \cite{Seidel01}. This proposition indeed states that $HF^*(\tau_l^{2m}(L'),L')$ is isomorphic to the cohomology of
$$CF(L',L')\oplus (a) \otimes CF(L',L)\oplus (a) \otimes (\varepsilon)\otimes CF(L',L)\oplus \cdots \leftarrow (a) \oplus \overbrace{(\varepsilon)\cdots (\varepsilon)}^{2m-1}\otimes CF(L',L),$$
where $(a)$ denotes the one dimensional vector space generated by $a$, the only intersection point between $L$ and $L'$and $(\varepsilon)$ the one dimensional vector space generated by $\varepsilon$, the generator corresponding to the top degree cohomology class in $HF(L',L')\cong H^*(L';\Z/2)$; the differential acts as $0$ on the first summand, as $\mu^2$ on the second one and on the $r^{th}$ summand by
$$\sum_{\substack{r=i+j\\j>1}}(\mu^i\otimes \Id^{\otimes j}+\Id^{\otimes j}\otimes\mu^i).$$
As it can be seen as the dual of $a$, we will denote $a^{\vee}$ the generator of $CF(L',L)$. Note that all the summands except the first one are one dimensional. For all $k\neq 2$, $\mu^k = 0$ and $\mu^2$ vanishes on all summands except the second one. The differential is then equal to zero on each summand \cite{Sei08}, except on the second one where we have $\mu^2(a,a^{\vee})=\varepsilon$. Consequently, we immediately deduce that the rank of $HF^*(\tau_l^{2m}(L'),L')$ is $2m$, which concludes the proof of this proposition.

\end{proof}

This will allow us to distinguish all the different powers of the Dehn-Seidel twist, except for the identity and the square of the Dehn-Seidel twist. Nethertheless, this will not be an issue.

\subsection{Connected components of the powers of the Dehn-Seidel twist}
Using \autoref{prop contL barcode} and its \autoref{Cor extension} and \autoref{prop rk cohom Dehn twist}, we will now prove \autoref{thm A} stated in the introduction. For the reader's convenience, we repeat it here.

\bigskip

\begin{thm}\label{thm tau not in Symp0}

Let $(M^{2n},\omega)$ be a $2n$-dimensionnal Liouville domain, $n$ even, $2c_1(M,\omega)=0$. Assume that M contains an $A_2$-configuration of Lagrangian spheres $(l,l')$.

Then, for all $k\in\Z\setminus \{0\}$, $\tau_l^{2k}$ is not in the connected component of the identity in $\overline{\Symp}(M,\omega)$.
\end{thm}

\begin{proof}
Let $k\in\N \setminus\{1\}$ and assume that $\tau_l^{2k}$ is in the connected component of the identity in $\overline{\Symp}(M,\omega)$. The first point of \autoref{thm barcodes connected} implies then that $B^k=\hat{B}(\tau_l^{2k}(L'),L')$ and $B^0=\hat{B}(L',L')$ are in the same connected component in $\hat{\calB}$, i.e. they have their semi-infinite bars in the same degree.

Moreover, since $HF(L',L')=\Z/2^{[0]}\oplus\Z/2^{[n]}$, \autoref{prop homologie barres infinies} implies that $B^0=\hat{B}(L',L')$ has only two semi-infinite bars, one in degree $0$ and one in degree $n$. Let us also recall that \autoref{prop rk cohom Dehn twist} gives, for all $k\in\N^*$,
$$\mathrm{rk}(HF^*(\tau_l^{2k}(L'),L'))=2k.$$

This means that $B^k$ cannot have the semi-infinite bars in the same degree as $B^0$. This contradicts the fact that $\tau_l^{2k}$ and the identity are in the same connected component of $\overline{\Symp}(M,\omega)$. It is immediate that $\tau_l^2$ is not in the connected component of the identity in $\overline{\Symp}(M,\omega)$. Indeed it was the case, $\tau_l^{4}$ would be as well.

For the same reason, we get that  for all $k\in\Z$, $\tau_l^{2k}$ is not in the connected component ofthe identity in $\overline{\Symp}(M,\omega)$ and thus concludes the proof of this theorem.
\end{proof}

The following statements correspond to \autoref{cor thm A} and \autoref{cor thm A Ham} stated in the introduction. For the reader's convenience, we repeat them here. The first one is a straightforward consequence of \autoref{thm tau not in Symp0}.

\begin{cor}\label{cor Dt isotopic}
Let $(M^{2n},\omega)$ be a $2n$-dimensional Liouville domain, $n$ even, $2c_1(M,\omega)=0$. Assume that M contains an $A_2$-configuration of Lagrangian spheres $(l,l')$.

Then, $\tau_l^2$ is an element of infinite order in $\pi_0(\overline{\Symp}(M,\omega))$. In particular, $\mathrm{MCG}^\omega(M,C^0)$ is non-trivial.
\end{cor}

As mentionned in the introduction, this result is an interesting example of $C^0$ symplectic rigidity in the sense that it shows a different behaviour from what happens outside the symplectic world. It comes indeed in contrast with the fact that the Dehn-Seidel twist is known to be of order $2$ or $4$ in $\pi_0(\mathrm{Homeo_c}(T^*\mathbb{S}^n))$ for all $n$ even to the work of Kauffman-Krylov \cite{KK05} and Krylov \cite{K07}.

Let us recall that, according to the discussion held in the introduction, \autoref{cor Dt isotopic} does not imply \autoref{thm tau not in Symp0}. Indeed, whether $\overline{\Symp}(M,\omega)$ is locally path-connected remains an open question.

The following theorem is also a consequence of \autoref{thm tau not in Symp0}. Indeed the subspace $\overline{\Ham}(M,\omega)\subset\overline{\Symp}(M,\omega)$ is connected as it is the closure of the connected space $\Ham(M,\omega)$.

\begin{thm}
Let $(M^{2n},\omega)$ be a $2n$-dimensional Liouville domain, $n$ even, $2c_1(M,\omega)=0$. Assume that M contains an $A_2$-configuration of Lagrangian spheres $(l,l')$.

Then, $\forall k\in\Z^*, \tau_{l}^{2k}$ does not belong to $\overline{\Ham}(M,\omega)$.
\end{thm}

\appendix
\section{Appendix: remarks for dimension $4$}\label{sec dim 4}
We present here some results in dimension $4$ that can be obtain without the use of barcodes. The proofs rely on Hind's result regarding the nearby Lagrangian Conjecture of \autoref{thm NLC S2} in the case of $T^*S^2$ \cite{H04}. If this result were true in higher dimension, we could apply the same arguments as the following ones to these higher dimensions. However, we get some slightly weaker resuls: in order to use Seidel's \autoref{thm Seidel Dtwist} we require an $A_3$-configuration instead of an $A_2$-configuration and we get results dealing path connected components instead of connected components.

\begin{thm}\label{thm Ham dim 4}

Let $(M^{4},\omega)$ be a compact connected $4$-dimensional submanifold with contact type boundary, $[\omega]=0$, $2c_1(M,\omega)=0$. Assume that M contains an $A_3$-configuration of Lagrangian spheres $(l,l',l_{\infty})$.

Then, $\forall k\in\Z^*, \tau_l^{2k}$ does not belong to $\overline{\Ham}(M,\omega)$.
\end{thm}
Using Hind's \autoref{thm NLC S2} in the case of $T^*S^2$ \cite{H04}, the proof is quite straightforward.
\begin{proof}
Let us assume that there is a sequence of Hamiltonian diffeomorphisms $(\varphi_n)_{n\in\N}$ of $M$ which $C^0$-converges to $\tau_l^2$. Then, for $N$ large enough, we have that $\varphi_N(L')$ is included in a Weinstein neighbourhood of $\tau_l^2(L')$. 

Moreover, $\tau_l^2(L')$ is a Lagrangian sphere and so its Weinstein neighbourhood is, by definition, symplectomorphic to a neighbourhood of the zero section in $T^{*}\mathbb{S}^2$. Consequently, we are under the condition of application of the Nearby Lagrangian Conjecture as in \autoref{thm NLC S2}, in the case of $T^{*}\mathbb{S}^2$.

We get that $\varphi_N(L')$ is Lagrangian isotopic to $\tau_l^2(L')$, which contradicts Seidel's result in \autoref{thm Seidel Dtwist}.
\end{proof}

The following theorem corresponds to \autoref{cor Dt isotopic} in dimension $4$. It is slightly weaker than the result of this corollary but we present its proof here because, as for \autoref{thm Ham dim 4}, its proof does not require the use of barcodes.

\begin{thm}\label{thm Symp dim 4}
Let $(M^{4},\omega)$ be a compact connected $4$-dimensionnal submanifold with contact type boundary, $[\omega]=0$, $2c_1(M,\omega)=0$. Assume that M contains an $A_3$-configuration of Lagrangian spheres $(l,l',l_{\infty})$.

Then, for all $k\in\Z$, $\tau_l^{2k}$ is not isotopic to the identity in $\overline{\Symp}(M,\omega)$.
\end{thm}

Note that none of \autoref{thm Ham dim 4} and \autoref{thm Symp dim 4} imply the other.

\begin{proof}
As above for \autoref{thm Ham dim 4}, this proof heavily relies on the proof of the nearby Lagrangian conjecture for $T^*S^2$ as in \autoref{thm NLC S2} \cite{H04}.

Let $k\in\Z$ and assume that $\tau_l^{2k}$ is isotopic to the identity in $\overline{\Symp}(M,\omega)$. This means that we can find a continuous path $(\varphi^t)_{t\in[0;1]}\subset\overline{\Symp}(M,\omega)$ such that $\varphi^0=\Id$ and $\varphi^1=\tau_l^{2k}$.

Since for all $t\in[0;1]$, $\varphi^t$ is in $\overline{\Symp}(M,\omega)$, we can find sequences $\varphi_{n}^t\in \Symp(M,\omega)$ such that
$$\forall t \in (0;1), \lim_{n\to \infty} \varphi_{n}^t = \varphi^t.$$

Let us choose a Weinstein neighbourhood $W(L')$ of $L'$ together with $\varepsilon>0$ such that, for all $\varphi\in\Symp(M,\omega)$, if $d_{C^0}(\varphi,\Id_M)<\varepsilon$, then $\varphi(L')\subset W(L')$.

The path $\varphi^t$ being continuous, we can find a finite sequence $(\varphi^{t_i})_{i\in \llbracket 0;N \rrbracket }\subset \overline{\Symp}(M,\omega)$ such that $\varphi^{t_0}= \Id$, $\varphi^{t_N}=\tau_l^{2k}$ and
$$\forall i \in \llbracket 1;N \rrbracket,\quad d_{C^0}(\varphi^{t_{i-1}},\varphi^{t_i})< \tfrac{\varepsilon}{3}.$$
Moreover, for each $(t_i)_{i\in\llbracket 0;N-1 \rrbracket}$, we can find $n_i$ such that $d_{C^0}(\varphi^{t_i},\varphi_{n_i}^{t_i})<\tfrac{\varepsilon}{3}.$ We choose $\varphi_0^{t_0}=\Id$ and $\varphi_N^{t_N}=\tau_l^{2k}$.

Consequently, we get a sequence $(\varphi_{n_i}^{t_i})_{i\in \llbracket 0;N \rrbracket }\subset \Symp(M,\omega)$ such that $\varphi_{n_0}^{t_0}= \Id_M$, $\varphi_{n_N}^{t_N}=\tau_l^{2k}$ which satisfies
$\forall i \in \llbracket 1;N \rrbracket$,
\begin{eqnarray*}
d_{C^0}((\varphi_{n_{i-1}}^{t_{i-1}})^{-1}\circ\varphi_{n_i}^{t_i},\Id_M)&\leq& d_{C^0}(\varphi_{n_{i-1}}^{t_{i-1}},\varphi_{n_i}^{t_i})\\
&\leq&d_{C^0}(\varphi_{n_{i-1}}^{t_{i-1}},\varphi^{t_{i-1}})+d_{C^0}(\varphi^{t_{i-1}},\varphi^{t_i})+d_{C^0}(\varphi^{t_{i}},\varphi_{n_i}^{t_i})\\
&<&\tfrac{\varepsilon}{3}+\tfrac{\varepsilon}{3}+\tfrac{\varepsilon}{3}=\varepsilon.
\end{eqnarray*}
Then \begin{eqnarray*}
\varphi_{n_i}^{t_i}(L')&=&\varphi_{n_{i-1}}^{t_{i-1}}\circ(\varphi_{n_{i-1}}^{t_{i-1}})^{-1}\circ\varphi_{n_i}^{t_i}(L')\\
&\subset&\varphi_{n_{i-1}}^{t_{i-1}}(W(L'))\cong W(\varphi_{n_i}^{t_i}(L')),
\end{eqnarray*}
where $W(\varphi_{n_i}^{t_i}(L'))$ denotes a Weinstein neighbourhood of $\varphi_{n_i}^{t_i}(L')$.

Applying now Hind's \autoref{thm NLC S2}, we obtain that for all $i \in \llbracket 1;N \rrbracket$, $\varphi_{n_i,t_i}(L')$ is Lagrangian isotopic to $\varphi_{n_{i-1},t_{i-1}}(L')$. Gluing these paths together, we finally get that $\tau_l^{2k}(L')$ is Lagrangian isotopic to $L'$. This contradicts Seidel's result of \autoref{thm Seidel Dtwist} and concludes this proof.
\end{proof}

\begin{rk}
We could have also proved this result using the second point of \autoref{thm barcodes connected} to construct a continuous path of barcodes between $\hat{B}(\tau_l^{2k}(L')',L')$ and $\hat{B}(L',L')$. This path together with \autoref{prop path inf bar} telling that the degree of the semi-infinite bars cannot change along a continuous path leads to a contradiction.
\end{rk}

\bibliographystyle{siam}
\bibliography{Morceaux/Bibli}

\begin{thebibliography}{10}

\bibitem{Ab10}
{\sc M.~Abouzaid}, {\em Nearby {L}agrangians with vanishing {M}aslov class are
  homotopy equivalent}, Invent. Math., 189 (2012), pp.~251--313.

\bibitem{AK18}
{\sc M.~Abouzaid and T.~Kragh}, {\em Simple homotopy equivalence of nearby
  {L}agrangians}, Acta Math., 220 (2018), pp.~207--237.

\bibitem{Arn95}
{\sc V.~I. Arnold}, {\em Some remarks on symplectic monodromy of {M}ilnor
  fibrations}, in The {F}loer memorial volume, vol.~133 of Progr. Math.,
  Birkh\"{a}user, Basel, 1995, pp.~99--103.

\bibitem{Aur14}
{\sc D.~Auroux}, {\em A beginner's introduction to {F}ukaya categories}, in
  Contact and symplectic topology, vol.~26 of Bolyai Soc. Math. Stud.,
  J\'{a}nos Bolyai Math. Soc., Budapest, 2014, pp.~85--136.

\bibitem{Bar94}
{\sc S.~A. Barannikov}, {\em The framed {M}orse complex and its invariants}, in
  Singularities and bifurcations, vol.~21 of Adv. Soviet Math., Amer. Math.
  Soc., Providence, RI, 1994, pp.~93--115.

\bibitem{BL14}
{\sc U.~Bauer and M.~Lesnick}, {\em Induced matchings of barcodes and the
  algebraic stability of persistence}, in Computational geometry ({S}o{CG}'14),
  ACM, New York, 2014, pp.~355--364.

\bibitem{BC07}
{\sc P.~Biran and O.~Cornea}, {\em Quantum structures for {L}agrangian
  submanifolds}, arXiv:0708.4221,  (2007).

\bibitem{BV18}
{\sc P.~Bubenik and T.~Vergili}, {\em Topological spaces of persistence modules
  and their properties}, J. Appl. Comput. Topol., 2 (2018), pp.~233--269.

\bibitem{BHS18}
{\sc L.~Buhovsky, V.~Humili\`ere, and S.~Seyfaddini}, {\em A {$C^0$}
  counterexample to the {A}rnold conjecture}, Invent. Math., 213 (2018),
  pp.~759--809.

\bibitem{BHS19}
{\sc L.~Buhovsky, V.~Humili\`ere, and S.~Seyfaddini}, {\em An {A}rnold-type
  principle for non-smooth objects}, arXiv:1909.07081,  (2019).

\bibitem{BO16}
{\sc L.~Buhovsky and E.~Opshtein}, {\em Some quantitative results in
  {$\mathcal{C}^0$} symplectic geometry}, Invent. Math., 205 (2016), pp.~1--56.

\bibitem{Carlson&al04}
{\sc G.~Carlsson, A.~Zomorodian, A.~Collins, and L.~Guibas}, {\em Persistence
  barcodes for shapes.}, International Journal of Shape Modeling, 11 (2005),
  pp.~149--188.

\bibitem{Chazal&al16}
{\sc F.~Chazal, V.~de~Silva, M.~Glisse, and S.~Oudot}, {\em The structure and
  stability of persistence modules}, SpringerBriefs in Mathematics, Springer,
  [Cham], 2016.

\bibitem{Che08}
{\sc A.~V. Chernavski\u{\i}}, {\em Local contractibility of the homeomorphism
  group of {$\mathbb{R}^n$}}, Tr. Mat. Inst. Steklova, 263 (2008),
  pp.~201--215.

\bibitem{Cof05}
{\sc J.~Coffey}, {\em Symplectomorphism groups and isotropic skeletons}, Geom.
  Topol., 9 (2005), pp.~935--970.

\bibitem{CB15}
{\sc W.~Crawley-Boevey}, {\em Decomposition of pointwise finite-dimensional
  persistence modules}, J. Algebra Appl., 14 (2015), pp.~1550066, 8.

\bibitem{DGE14}
{\sc G.~Dimitroglou~Rizell and J.~D. Evans}, {\em Unlinking and unknottedness
  of monotone {L}agrangian submanifolds}, Geom. Topol., 18 (2014),
  pp.~997--1034.

\bibitem{RGI16}
{\sc G.~Dimitroglou~Rizell, E.~Goodman, and A.~Ivrii}, {\em Lagrangian isotopy
  of tori in {$S^2\times S^2$} and {$\mathbb{C}P^2$}}, Geom. Funct. Anal., 26
  (2016), pp.~1297--1358.

\bibitem{Edel&al00}
{\sc H.~Edelsbrunner, D.~Letscher, and A.~Zomorodian}, {\em Topological
  persistence and simplification}, in 41st {A}nnual {S}ymposium on
  {F}oundations of {C}omputer {S}cience ({R}edondo {B}each, {CA}, 2000), IEEE
  Comput. Soc. Press, Los Alamitos, CA, 2000, pp.~454--463.

\bibitem{El87}
{\sc Y.~M. Eliashberg}, {\em A theorem on the structure of wave fronts and its
  application in symplectic topology}, Funktsional. Anal. i Prilozhen., 21
  (1987), pp.~65--72, 96.

\bibitem{EP09}
{\sc M.~Entov and L.~Polterovich}, {\em Rigid subsets of symplectic manifolds},
  Compos. Math., 145 (2009), pp.~773--826.

\bibitem{EV11}
{\sc J.~D. Evans}, {\em Symplectic mapping class groups of some {S}tein and
  rational surfaces}, J. Symplectic Geom., 9 (2011), pp.~45--82.

\bibitem{Fat80}
{\sc A.~Fathi}, {\em Structure of the group of homeomorphisms preserving a good
  measure on a compact manifold}, Ann. Sci. \'{E}cole Norm. Sup. (4), 13
  (1980), pp.~45--93.

\bibitem{Flo87}
{\sc A.~Floer}, {\em Morse theory for fixed points of symplectic
  diffeomorphisms}, Bull. Amer. Math. Soc. (N.S.), 16 (1987), pp.~279--281.

\bibitem{Floer88}
\leavevmode\vrule height 2pt depth -1.6pt width 23pt, {\em Morse theory for
  {L}agrangian intersections}, J. Differential Geom., 28 (1988), pp.~513--547.

\bibitem{Flo89}
\leavevmode\vrule height 2pt depth -1.6pt width 23pt, {\em Witten's complex and
  infinite-dimensional {M}orse theory}, J. Differential Geom., 30 (1989),
  pp.~207--221.

\bibitem{FOOO09}
{\sc K.~Fukaya, Y.-G. Oh, H.~Ohta, and K.~Ono}, {\em Lagrangian intersection
  {F}loer theory: anomaly and obstruction. {P}art {I}}, vol.~46 of AMS/IP
  Studies in Advanced Mathematics, American Mathematical Society, Providence,
  RI; International Press, Somerville, MA, 2009.

\bibitem{FSI07}
{\sc K.~Fukaya, P.~Seidel, and I.~Smith}, {\em Exact {L}agrangian submanifolds
  in simply-connected cotangent bundles}, Invent. Math., 172 (2008), pp.~1--27.

\bibitem{H04}
{\sc R.~Hind}, {\em Lagrangian spheres in {$S^2\times S^2$}}, Geom. Funct.
  Anal., 14 (2004), pp.~303--318.

\bibitem{Hof90}
{\sc H.~Hofer}, {\em On the topological properties of symplectic maps}, Proc.
  Roy. Soc. Edinburgh Sect. A, 115 (1990), pp.~25--38.

\bibitem{HLS15}
{\sc V.~Humili\`ere, R.~Leclercq, and S.~Seyfaddini}, {\em Coisotropic rigidity
  and {$C^0$}-symplectic geometry}, Duke Math. J., 164 (2015), pp.~767--799.

\bibitem{HLS16}
\leavevmode\vrule height 2pt depth -1.6pt width 23pt, {\em Reduction of
  symplectic homeomorphisms}, Ann. Sci. \'{E}c. Norm. Sup\'{e}r. (4), 49
  (2016), pp.~633--668.

\bibitem{KK05}
{\sc L.~H. Kauffman and N.~A. Krylov}, {\em Kernel of the variation operator
  and periodicity of open books}, Topology Appl., 148 (2005), pp.~183--200.

\bibitem{Kaw19}
{\sc Y.~Kawamoto}, {\em On ${C}^0$-continuity of the spectral norm on
  non-symplectically aspherical manifolds}, arXiv:1905.07809,  (2019).

\bibitem{Kea12}
{\sc A.~M. Keating}, {\em Dehn twists and free subgroups of symplectic mapping
  class groups}, J. Topol., 7 (2014), pp.~436--474.

\bibitem{KovSeid00}
{\sc M.~Khovanov and P.~Seidel}, {\em Quivers, {F}loer cohomology, and braid
  group actions}, J. Amer. Math. Soc., 15 (2002), pp.~203--271.

\bibitem{KS18}
{\sc A.~Kislev and E.~Shelukhin}, {\em Bounds on spectral norms and barcodes},
  arXiv:1810.09865,  (2018).

\bibitem{Kr13}
{\sc T.~Kragh}, {\em Parametrized ring-spectra and the nearby {L}agrangian
  conjecture}, Geom. Topol., 17 (2013), pp.~639--731.
\newblock With an appendix by Mohammed Abouzaid.

\bibitem{K07}
{\sc N.~A. Krylov}, {\em Relative mapping class group of the trivial and the
  tangent disk bundles over the sphere}, Pure Appl. Math. Q., 3 (2007),
  pp.~631--645.

\bibitem{LaudSik94}
{\sc F.~Laudenbach and J.-C. Sikorav}, {\em Hamiltonian disjunction and limits
  of {L}agrangian submanifolds}, Internat. Math. Res. Notices,  (1994), pp.~161
  ff., approx. 8 pp.

\bibitem{LR10}
{\sc F.~Le~Roux}, {\em Simplicity of {${\rm
  Homeo}(\mathbb{D}^2,\partial\mathbb{D}^2,{\rm Area})$} and fragmentation of
  symplectic diffeomorphisms}, J. Symplectic Geom., 8 (2010), pp.~73--93.

\bibitem{LRSV18}
{\sc F.~Le~Roux, S.~Seyfaddini, and C.~Viterbo}, {\em Barcodes and
  area-preserving homeomorphisms}, 2018.

\bibitem{Lec07}
{\sc R.~Leclercq}, {\em Spectral invariants in {L}agrangian {F}loer theory}, J.
  Mod. Dyn., 2 (2008), pp.~249--286.

\bibitem{LZ18}
{\sc R.~Leclercq and F.~Zapolsky}, {\em Spectral invariants for monotone
  {L}agrangians}, J. Topol. Anal., 10 (2018), pp.~627--700.

\bibitem{LLW15}
{\sc J.~Li, T.-J. Li, and W.~Wu}, {\em The symplectic mapping class group of
  {$\mathbb CP^2\#n\overline{\mathbb {CP}^2}$} with {$n \le 4$}}, Michigan
  Math. J., 64 (2015), pp.~319--333.

\bibitem{MS3rd}
{\sc D.~McDuff and D.~Salamon}, {\em Introduction to symplectic topology},
  Oxford Graduate Texts in Mathematics, Oxford University Press, Oxford,
  third~ed., 2017.

\bibitem{MVZ12}
{\sc A.~Monzner, N.~Vichery, and F.~Zapolsky}, {\em Partial quasimorphisms and
  quasistates on cotangent bundles, and symplectic homogenization}, J. Mod.
  Dyn., 6 (2012), pp.~205--249.

\bibitem{Oh95}
{\sc Y.-G. Oh}, {\em Floer cohomology of {L}agrangian intersections and
  pseudo-holomorphic disks. {III}. {A}rnold-{G}ivental conjecture}, in The
  {F}loer memorial volume, vol.~133 of Progr. Math., Birkh\"{a}user, Basel,
  1995, pp.~555--573.

\bibitem{Oh04}
{\sc Y.-G. Oh}, {\em Construction of spectral invariants of {H}amiltonian paths
  on closed symplectic manifolds}, in The breadth of symplectic and {P}oisson
  geometry, vol.~232 of Progr. Math., Birkh\"{a}user Boston, Boston, MA, 2005,
  pp.~525--570.

\bibitem{Oh15}
\leavevmode\vrule height 2pt depth -1.6pt width 23pt, {\em Symplectic topology
  and {F}loer homology.}, vol.~28 of New Mathematical Monographs, Cambridge
  University Press, Cambridge, 2015.
\newblock Symplectic geometry and pseudoholomorphic curves.

\bibitem{MulOh07}
{\sc Y.-G. Oh and S.~M\"{u}ller}, {\em The group of {H}amiltonian
  homeomorphisms and {$C^0$}-symplectic topology}, J. Symplectic Geom., 5
  (2007), pp.~167--219.

\bibitem{PolShel14}
{\sc L.~Polterovich and E.~Shelukhin}, {\em Autonomous {H}amiltonian flows,
  {H}ofer's geometry and persistence modules}, Selecta Math. (N.S.), 22 (2016),
  pp.~227--296.

\bibitem{Sch00}
{\sc M.~Schwarz}, {\em On the action spectrum for closed symplectically
  aspherical manifolds}, Pacific J. Math., 193 (2000), pp.~419--461.

\bibitem{Seidel99}
{\sc P.~Seidel}, {\em Lagrangian two-spheres can be symplectically knotted}, J.
  Differential Geom., 52 (1999), pp.~145--171.

\bibitem{Seidel00}
\leavevmode\vrule height 2pt depth -1.6pt width 23pt, {\em Graded {L}agrangian
  submanifolds}, Bull. Soc. Math. France, 128 (2000), pp.~103--149.

\bibitem{Seidel01}
\leavevmode\vrule height 2pt depth -1.6pt width 23pt, {\em A long exact
  sequence for symplectic {F}loer cohomology}, Topology, 42 (2003),
  pp.~1003--1063.

\bibitem{Sei08}
\leavevmode\vrule height 2pt depth -1.6pt width 23pt, {\em Fukaya categories
  and {P}icard-{L}efschetz theory}, Zurich Lectures in Advanced Mathematics,
  European Mathematical Society (EMS), Z\"{u}rich, 2008.

\bibitem{SeiPhD}
{\sc P.~Seidel and D.~Phil}, {\em Floer homology and the symplectic isotopy
  problem}, PhD thesis, Citeseer, 1997.

\bibitem{ST01}
{\sc P.~Seidel and R.~Thomas}, {\em Braid group actions on derived categories
  of coherent sheaves}, Duke Math. J., 108 (2001), pp.~37--108.

\bibitem{Sey13}
{\sc S.~Seyfaddini}, {\em {$C^0$}-limits of {H}amiltonian paths and the
  {O}h-{S}chwarz spectral invariants}, Int. Math. Res. Not. IMRN,  (2013),
  pp.~4920--4960.

\bibitem{Shel18}
{\sc E.~Shelukhin}, {\em {V}iterbo conjecture for {Z}oll symmetric spaces},
  arXiv:1811.05552,  (2018).

\bibitem{Shel19}
\leavevmode\vrule height 2pt depth -1.6pt width 23pt, {\em {S}ymplectic
  cohomology and a conjecture of {V}iterbo}, arXiv:1904.06798,  (2019).

\bibitem{sik94}
{\sc J.-C. Sikorav}, {\em Some properties of holomorphic curves in almost
  complex manifolds}, in Holomorphic curves in symplectic geometry, vol.~117 of
  Progr. Math., Birkh\"{a}user, Basel, 1994, pp.~165--189.

\bibitem{Ush11}
{\sc M.~Usher}, {\em Boundary depth in {F}loer theory and its applications to
  {H}amiltonian dynamics and coisotropic submanifolds}, Israel J. Math., 184
  (2011), pp.~1--57.

\bibitem{Ush14}
\leavevmode\vrule height 2pt depth -1.6pt width 23pt, {\em Hofer's metrics and
  boundary depth}, Ann. Sci. \'{E}c. Norm. Sup\'{e}r. (4), 46 (2013),
  pp.~57--128 (2013).

\bibitem{UZ15}
{\sc M.~Usher and J.~Zhang}, {\em Persistent homology and {F}loer-{N}ovikov
  theory}, Geom. Topol., 20 (2016), pp.~3333--3430.

\bibitem{Vit92}
{\sc C.~Viterbo}, {\em Symplectic topology as the geometry of generating
  functions}, Math. Ann., 292 (1992), pp.~685--710.

\bibitem{Wu13}
{\sc W.~Wu}, {\em Exact {L}agrangians in {$A_n$}-surface singularities}, Math.
  Ann., 359 (2014), pp.~153--168.

\end{thebibliography}

\end{document}